\documentclass{amsart}

\usepackage[hmargin=3.5cm,top=3cm,bottom=3cm,includehead]{geometry}

\usepackage{mathabx} 

\usepackage{amsmath, amssymb, amsthm, mathrsfs}
\usepackage{graphicx,enumitem} 
\usepackage{cancel}
 
\usepackage{xcolor}
\usepackage[pdfborder={0 0 0}]{hyperref}

\newcommand\R{{\mathbb R}}

\def\BB{{\mathcal B}}

\def\DD{{\mathcal D}}

\def\HH{{\mathcal H}}
\def\II{{\mathcal I}}

\def\NN{{\mathcal N}}
\def\OO{{\mathcal O}}

\def\CCC{{\mathscr C}}

\def\LLL{{\mathscr L}}

\def\NNN{{\mathscr N}}

\def\RRR{{\mathscr R}}

\def\UUU{{\mathscr U}}
\def\VVV{{\mathscr V}}

\def\XXX{{\mathscr X}}

 \def\Wloc{W_{\mathrm{loc}}}

\def\Lloc{L_{\mathrm{loc}}}

\def\cc{{\frak c}}

\def\eps{{\varepsilon}}
\def\wto{{\,\rightharpoonup\,}}
\def\Lloc{L_{\rm loc}}

\def\fraka{{{\mathfrak {a}}}}
\def\frakb{{\frak b}}

\newtheorem{theo}{Theorem}[section]
\newtheorem{prop}[theo]{Proposition}
\newtheorem{lem}[theo]{Lemma}
\newtheorem{cor}[theo]{Corollary}
\newtheorem*{thm*}{Theorem}
\theoremstyle{remark}
\newtheorem{rem}[theo]{Remark}

\newtheorem*{ex*}{Example}
\theoremstyle{definition}
\newtheorem{definition}[theo]{Definition}

\numberwithin{equation}{section}
\newcommand{\be}{\begin{equation}}
\newcommand{\ee}{\end{equation}}
\newcommand{\ba}{\begin{aligned}}
\newcommand{\ea}{\end{aligned}}
\newcommand{\beqn}{\begin{equation}}
\newcommand{\eeqn}{\end{equation}}
\newcommand{\bear}{\begin{eqnarray}}
\newcommand{\eear}{\end{eqnarray}}
\newcommand{\bean}{\begin{eqnarray*}}
\newcommand{\eean}{\end{eqnarray*}}

\newcommand{\Black}{\color{black}}

\title[Existence of solutions to the VCk equation]
{Existence of solutions to the Voltage-Conductance kinetic equation in a general conductivity regime
}

\date{\today}

\author[C. Fonte Sanchez]{C. Fonte Sanchez}

\author[S. Mischler]{S. Mischler}

\author[D. Salort]{D. Salort}

\address[C. Fonte Sanchez]{Institute for individualized immunotherapies,
iiT Research \& Development GmbH, Strubergasse 18/3, 5020 Salzburg
}
\email{c.fonte@iit-research.at}

\address[S.~Mischler]{Centre de Recherche en Math\'ematiques de
  la D\'ecision (CEREMADE, CNRS UMR 7534),
  Universit\'es PSL \& Paris-Dauphine, Place de Lattre de
  Tassigny, 75775 Paris 16, France \& Institut Universitaire de France (IUF)}
\email{mischler@ceremade.dauphine.fr}

\address[D. Salort] {Sorbonne Universit{\'e}, CNRS,   Universit\'{e} de Paris,  Laboratoire Jacques-Louis Lions, F-75005 Paris, Institut Universitaire de France (IUF), Supported by fondation Cino et Simone Del Duca.}\email{delphine.salort@sorbonne-universite.fr}
 
 

\begin{document}

\begin{abstract} We establish the existence of a weak solution to the Voltage-Conductance kinetic equation for any conductivity  parameter and any reasonable initial datum, and in particular without smallness condition as it was the case in \cite{CFS+SM}. The proof is based on some a priori estimates  mainly drawn from \cite{MR3177631}, some trace results adapted from \cite{MR2721875} 
and a DiPerna-Lions weak stability argument.

\end{abstract}

\maketitle

\tableofcontents
 
\section{Introduction}

\subsection{The VCk equation} 
In this paper, we are concerned with the existence of solutions issue  for  the Voltage-Conductance kinetic (VCk) equation modeling the time evolution of a neuron network, see \cite{MR3177631}. 
We thus consider the VCk equation  
\beqn\label{eq:VCk}
\partial_t F  + \partial_v(JF) + \partial_y(K_F F) - a_F \partial_{yy}^2 F = 0
\quad\text{in}\quad (0,\infty) \times \OO
\eeqn
on the probability distribution $F =  F(t,v,y) \ge 0$ of neurons which have voltage $v \in (0,v_F)$ and conductance $y > 0$ at time $t \ge 0$.
Here, we use  the shorthand $(v,y) \in \OO := (0,v_F) \times (0,\infty) $ and  the coefficients are given by 
 \bean
 J &:=& y (v_E-v) - y_L v, 
 \\
 K_F &:=& y_* + \cc\, \NN_F - y, 
 \\
a_F &:=& a_* + \cc^2 \NN_F , 
\eean
 with $v_E > v_F > 0$,  $y_L,a_*, y_* > 0$, $\cc \ge 0$. 
  For $F  : \bar\OO \to \R$, the global firing rate is defined as the current generated by all the  spiking neurons,  that is $\NN_F:= \NN(\gamma_+ F)$ with 
 \beqn
 \label{eq:defNNF}
 \NN(\gamma_+ F) (t)  :=    \int_0^\infty J_+(y) (\gamma_+ F) (t,y) dy, 
 \eeqn
where $\gamma_+ F$ is the trace of $F$ on the outgoing part of the boundary that we will define below, see \eqref{eq:defgamma1+}.  
It is worth noticing that the equation's nonlinearity, driven by the term $\NN_F$, is controlled by the connectivity parameter $\cc$. 
In fact, the equation becomes linear when there is no connection between neurons, which corresponds to $\cc = 0$. 
It is also worth noticing that, similarly to classical kinetic models, the noise term in the VCk equation acts exclusively on the conductance variable.

\smallskip
 We complement the VCk equation \eqref{eq:VCk} with  an initial condition
\beqn\label{eq:VCkt=0}
F(0,\cdot) = F_0 \ge 0\quad\text{in}\quad \OO,
 \eeqn
specifying the probability distribution at time zero. Additionally,   defining by $\gamma F$ the trace of the function $F$ on the boundary of the  domain,
 the following boundary conditions are imposed: 
  \bear \label{eq:VCktBd1}
&& \gamma F = 0  \ \hbox{ on } \   (0,T) \times \Sigma_1 , \\
\label{eq:VCktBd2}
&& (J \gamma F)(t,0,y) =  (J \gamma F) (t,v_F,y)   \ \hbox{ for any } \  t \in (0,T), \ y \in (y_F,\infty),  \\
\label{eq:VCktBd0}
&& K_F   \gamma F - a_F \partial_y \gamma F = 0 \ \hbox{ on } \  (0,T) \times \Sigma_0 , 
 \eear
 where $T \in (0,\infty]$ and we define $y_F := y_L v_F/(v_E-v_F)$, 
 \bean
 \Sigma_0 := (0,v_F) \times \{ 0 \}, \quad
  \Sigma_1 = (\{ 0 \} \cup \{v_F \}) \times (0,y_F). 
 \eean
The first condition \eqref{eq:VCktBd1} ensures that there is no net movement of the probability density across the boundary at low conductance values. In \eqref{eq:VCktBd2}, note that 
$y_F$ is the solution to the equation $J(v_F,y)=0$, and more precisely, it is the point where the function $y\mapsto J(  v_F,y)$ changes of sign. For values of $y\leq y_F$, the vector field $J$ is negative near $v_F$ so that the trajectories move away from the boundary at $v_F$, i.e., for conductance values smaller than $y_F$,  
the neuron does not receive enough current to spike. On the other hand,  as $y$ pass the value $y_F$ the voltage of the neuron may reach $v_F$, spiking. In that case, it is reset to zero instantaneously, from where the second boundary condition for $y>y_F$. In \eqref{eq:VCktBd0}, the condition on $(0,T) \times \Sigma_0$ ensures mass conservation. 
 
 \smallskip
 We note $\Sigma_2 := (\{ 0 \} \cup \{v_F \}) \times (y_F,\infty)$, $\Sigma_{12} :=  (\{ 0 \} \cup \{v_F \}) \times (0,\infty)$,  $\Sigma := \Sigma_{12} \cup \Sigma_0$, $\Gamma = (0,T) \times \Sigma$, $\Gamma_i = (0,T) \times \Sigma_i$, $i=0,1,2,12$, and $n_v := + 1$ for $v = v_F$, $n_v := - 1$ for $v=0$, 
 $n_0 := -1$. For further references, we also denote 
$$
\Sigma_i^{\pm} := \{(v,y) \in \Sigma_i; \, \pm J n_v > 0 \}, 
$$
so that 
$$
\Sigma_1^+ = \emptyset, \  \Sigma_1^- = \Sigma_1, \ \Sigma_2^- = \{0\} \times (y_F,\infty), \  \Sigma_2^+ = \{v_F\} \times (y_F,\infty),
$$
and we next define similarly $\Sigma_{12}^- = \Sigma_{1}^- \cup \Sigma_2^-$, $\Gamma_i^\pm := (0,T) \times \Sigma_i^\pm$. 
We will sometime summarize \eqref{eq:VCktBd0}-\eqref{eq:VCktBd1}-\eqref{eq:VCktBd2} with the shorthand 
\beqn\label{eq:VCkBd}
\RRR_F \gamma F = 0 \quad\text{on}\quad \Gamma. 
 \eeqn
We will also sometime use the shorthands
\beqn\label{eq:defgamma1+}
\gamma_1 F = \gamma F {\bf 1}_{\Sigma_1}, \quad \gamma_\pm F = \gamma F {\bf 1}_{\Sigma^\pm_2},
\eeqn
what in particular explain the definition \eqref{eq:defNNF}. Similarly we may use the shorthands  
\bean
&&J_+(y) := J(v_F,y) {\bf 1}_{y > y_F} = (v_E-v_F) (y-y_F)_+  
\\
&&J_-(y) := J(0,y) {\bf 1}_{y > y_F} = y v_E.
\eean
 

\smallskip
 The two important general properties of the model are that (at least formally) any solution is mass conservative, that is 
\beqn\label{eq:MassCons}
\langle F(t,\cdot) \rangle = \langle F_0 \rangle, \quad \forall \, t \ge 0, \quad \langle F \rangle :=  \int_\OO F dvdy, 
\eeqn
  and it conserves positivity, that is
\beqn\label{eq:PositiveCons}
F(t,\cdot) \ge 0, \quad \forall \, t \ge 0. 
\eeqn
For convenience, we may and will assume that $F_0$ is a probability measure and thus also is $F(t,\cdot)$ for any $t \ge 0$.   

\subsection{Main results}


Our main result establishes the existence of weak solutions to the  Voltage-Conductance kinetic (VCk) equation.

\begin{theo}\label{theo-Exists} 
For any initial datum $ F_0 \ge 0$ such that $F_0 (1+ y^2 + |\log F_0|) \in L^1(\OO)$ with unit mass ($\langle F_0 \rangle = 1$)  and any connectivity parameter $\cc \ge 0$, there exists at least one  solution $0 \le F \in C([0,\infty);L^1(\OO))$ 
to the Voltage-Conductance kinetic (VCk) equation   \eqref{eq:VCk},  \eqref{eq:defNNF}, \eqref{eq:VCkt=0}, \eqref{eq:VCktBd0}, \eqref{eq:VCktBd1},  \eqref{eq:VCktBd2} which furthermore satisfies \eqref{eq:MassCons} and the natural bounds.
\end{theo} 

We will make more precise the meaning in the next sections. 
The {\it natural bounds} will be established in Section~\ref{sec:A priori estim}. 
The definition of solutions will be given in Section~\ref{sec:def&stab} with the help of the trace theorems established in Section~\ref{sec:TracePb}. 
%
%
This result improves the existence result in \cite{CFS+SM}, where some restrictions on the connectivity parameter $\cc \ge 0$ and on the initial datum $F_0$ are imposed (a combination of the two must be small enough). The drawback is that in our present framework, the solutions are not proved to be uniformly bounded (in time). We also refer to \cite{MR4548856,MR4770834,CFS+SM}, where the stability issue of the VCk equation is addressed. 

\smallskip
The difficulty and the main mathematical interest come from the weak regularity of the mean flux term $ \NN(\gamma_+ F)$ which is involved in the nonlinearity nature of the equation. 

\smallskip\smallskip
Let us briefly describe the organization of the paper and the strategy of the proof, which follows a classical line since (at least) DiPerna and Lions work \cite{MR1014927}.

In  Section~\ref{sec:A priori estim}, we first collect some {\it a priori  estimates} for the solutions to the VCk equations. Most of them have been yet established in \cite{MR3177631}. These  {\it a priori  estimates} make possible to give a sense to a weak formulation of the equation and provide some weak compactness for both sequence of solutions and the associated sequence of trace functions. 

In  Section~\ref{sec:TracePb}, we present a trace theory adapted to our framework by carrying on the renormalized theory of DiPerna-Lions 
\cite{MR972541,MR1022305,MR1014927} and the associated  trace theory first developed in \cite{MR1765137,MR1776840,MR2721875}, and  recently revisited  in \cite{CM-Landau**,CGMM**}. The novelty here comes from the fact that the diffusion coefficient $a_F$ is not a constant but a $L^1(0,T)$ function. It is worth emphasizing that our proof simplifies some of the arguments used in the similar result established in  \cite{MR2721875}. 

In Section~\ref{sec:def&stab}, we introduce the precise definition of solutions we deal with. We next present a fundamental stability result in the spirit of DiPerna-Lions theory \cite{MR972541,MR1014927,MR1776840,MR2721875}, establishing that we may pass to the limit in any sequence of solutions which satisfy uniformly the {\it natural bounds.} 
 This result yet contains the most important difficulties that may be encountered during the  proof of the existence result. 
The proof revisits and refines some arguments already used in \cite{CFS+SM} and takes advantage of a key elementary convergence result (Lemma~\ref{lem:AB2}) for passing to the limit in a sequence of products of functions. 

Section~\ref{sec:existence} is devoted to the proof of Theorem~\ref{theo-Exists}, that we split into three steps. In Section~\ref{sec:LinearPb}, 
we just summarize some results useful for the linear problem associated to the VCk equation. In Section~\ref{sec:truncated-pb}, we introduce a truncated problem and establish the existence of solutions to this one, which will provide a sequence of approximated solutions in the next step. The truncated problem slightly differs from the one used in \cite{CFS+SM}. For the present one, the proof of the existence is based on the application of the Schauder fixed point  Theorem (instead of the Tykonov fixed point Theorem in \cite{CFS+SM}), what requires to be able to prove the strong compactness of the sequence of the out-coming trace associated to  a bounded sequence of solutions, generalizing and  simplifying the proof of a similar result established in \cite{MR2721875}. 
Section  \ref{subsec:passingTOlimit} is finally devoted to passing to the limit in the sequence of approximated solutions. Because of the choice of the truncation problem, the sequence automatically satisfies the estimates of Section~\ref{sec:A priori estim} and we may pass to the limit and thus conclude to the existence by (almost) using the stability result of Section~\ref{sec:def&stab}.

\subsection{Acknowledgements and AI tool disclosure}

The authors thank the ANR project ChaMaNe (ANR-19-CE40-0024) for the organization of a neuroscience workshop in Ile-Rousse, February 2025, where this investigation project started, as well as the 
LJK, university of Grenoble, for its kind hospitality.  D. Salort is supported by the fundation Cino and Simone Del Duca, Institut de France. The Claude AI has been used to help proofread typos. It also suggested a classical convergence result which enabled us to make slightly shorter and  simpler  the proof of the trace Theorem~\ref{theo:renormalization}. Apart from that,  the text in this paper was fully generated by the authors.

 \medskip
\section{A priori estimates} 

\label{sec:A priori estim}
 
In this section, we present some a priori estimates (formally) satisfied by any solution to the nonlinear VCk equation. 
It complements the two pieces of  information about positivity and mass conservation already mentioned in the introductory section. 
All together, these estimates  provide the natural functional framework in which we will   work in the sequel. 
Because we will need the same estimates for the truncated problem in the last section, we rather consider the more general (possibly nonlinear) equation 
\beqn\label{eq:Gal-eq-for-APB}
\partial_t F  + \partial_v(JF) + \partial_y(\frak K F) - \fraka \partial_{yy}^2 F = 0, 
\eeqn
with $\frak K := \frakb - y$, $\frak b \lesssim   \frak a$, $0 \le \frak b \le \frak C (1+\NN_F)$, $ a_*\leq \frak a \le \frak C (1+\NN_F)$  for some constants $\frak C > 0$,  that we complement with the boundary conditions \eqref{eq:VCktBd1}, \eqref{eq:VCktBd2} and 
\beqn\label{eq:VCktBd1-linear}
 \frak K \gamma F - \fraka  \partial_y \gamma F = 0 \ \hbox{ on } \ \Gamma_0.
\eeqn
Of course, that equation includes the VCk equation presented in the introduction by just taking $\fraka := a_F$ and $\frakb := y_* + \frak c \NN_F$.

 \smallskip
  In the sequel, we consider a nice solution $F \ge 0$ to the equations \eqref{eq:Gal-eq-for-APB}, \eqref{eq:VCktBd1}, \eqref{eq:VCktBd2}, \eqref{eq:VCktBd1-linear}, 
 and we establish several a priori estimates. The full rigorous justification needs more material and will be given in Section~\ref{sec:existence}. 
In this more general framework, we start reproving the moment estimate established in \cite[Theorem~4]{MR3177631}. 
 
 \begin{prop}\label{prop:APB1} For $k \in \{ 1,2\}$  and $T >0$, there holds
\beqn\label{eq:APB1}
 \sup_{[0,T]} \int_\OO F \langle y \rangle^k dydv + \int_0^T\int_0^\infty (J_+ \gamma_+ F) (t,y) \langle y \rangle^{k-1} dy dt \le C\Bigl(T,\int_\OO F_0 \langle y \rangle^k dydv\Bigr), 
\eeqn
where $\langle y \rangle := 1 + y$ (with no risk of confusion with the definition of $\langle \cdot \rangle$ in \eqref{eq:MassCons}). 
 \end{prop}

The proof draws on ideas from both the proofs of \cite[Theorem~4]{MR3177631} and \cite[Lemma~2.1]{CFS+SM}.

\begin{proof}[Proof of Proposition~\ref{prop:APB1}]

We denote
$$
V := \int_\OO v F dvdy, \quad Y_k := \int_\OO  y^k  F dvdy,  \quad Y^* := \int_\OO \phi_1 F dvdy, \quad Z :=  \int_\OO vy F dvdy,
$$
for any $k \ge 0$ and  for some $\phi_1 \in C^2(\R_+)$  such that  $\phi'_1 \ge 0$ on $\R_+$, $ \phi_1(y) = \tfrac12 y_F$ for any $ y \in [0,\tfrac13 y_F]$ and $\phi_1(y) = y$ for any $y \ge y_F$, $\phi_1(y) \ge y$ for any $y \in [0,y_F]$, 
in particular $  \phi_1 \le y + y_F$. 

\smallskip\noindent
{\it Step~1. Mass conservation.} We first compute  
\bean
{d \over dt} Y_0 
&=&
- \int_\OO   \partial_v(JF) - \int_\OO  \partial_y( \frak K F - \fraka \partial_{y} F)
\\
&=&
- \int_{\Sigma_{12}} \nu J\gamma F -  \int_{\Sigma_0} (  \frak K \gamma F - \fraka \partial_{y} \gamma F) = 0, 
\eean
by using the equation  \eqref{eq:Gal-eq-for-APB} in the first line, the Green formula in the second equality and the boundary condition in the last equality. 
That  proves the  mass conservation \eqref{eq:MassCons}, so that $Y_0(t) \equiv 1$.

\smallskip\noindent
{\it Step~2. First moment and flux.} We  similarly compute  
\bean
{d \over dt} V &=&    \int_\OO JF  - \int_{\Sigma^+_{2}} v J\gamma F
\\
&=& v_E Y_1 - Z - y_L V -  v_F\NN_F
\le    v_E Y_1 -  v_F\NN_F, 
\eean
where we have used the boundary condition in the first line, the very definitions of $J$ and $\NN_F$ in the second line and the positivity of the two terms we throw out in the third line.
We next compute 
\bean
{d \over dt} Y^* &=& \int_\OO F (\frakb - y) \phi'_1  - \int_\OO \fraka \partial_{y} F  \phi'_1 
\\
&=& \frakb \langle F \phi'_1 \rangle  -  \langle F y \phi'_1 \rangle  +  \fraka \langle F  \phi''_1 \rangle, 
\eean
where we have used the boundary condition in the first line and one more integration by part in the second line. 

From the very definition of $\phi_1$, there exist $C_1\ge 1$ such that 
$$ 
1- {\bf 1}_{y \le y_F} 
\le \phi'_1 \le C_1, \quad \phi''_1 \le C_1, 
$$
from what we deduce 
\bean 
{d \over dt} Y^* \le \frakb C_1 Y_0  + y_F Y_0 - Y_1   +  \fraka C_1 Y_0. 
\eean
As a consequence, for any $\eps > 0$, gathering the two differential estimates, we have 
\bean
{d \over dt} (V+\eps Y_1^*) 
\le
  (v_E - \eps  ) Y_1+ \eps  y_F Y_0 +\eps 2 \frak C C_1 Y_0  + (\eps 2 \frak C C_1 Y_0 - v_F)  \NN_F. 
\eean
Choosing now $\eps \in (0,1)$ small enough so that $\eps 2\frak C C_1 \le v_F /2$ and $\eps   \le v_E$, we deduce that 
\bean
{d \over dt} (V+\eps Y_1^*) + \frac {v_F}2 \NN_F  
\le  (\eps^{-1} v_E - 1) (V+\eps Y_1^*) + \eps y_F   +\eps 2 \frak C C_1 ,
\eean
and next, thanks to the Gronwall lemma, 
$$
 \sup_{[0,T]} \int F \langle y \rangle dvdy + \int_0^T  \NN_F  dt \le C e^{CT} \int F_0 \langle y \rangle dvdy, 
$$
for a constant $C = C(v_F, \eps, v_E,   \frak C, C_1)$, 
what is nothing but \eqref{eq:APB1} for $k = 1$. 

\smallskip\noindent
{\it Step~3. Second moment.}
Similarly as above, we now compute 
\bean
\frac12{d \over dt} Y_2  
&=&\int_\OO ( \frak K F - \fraka \partial_{y} F) y 
\\
&=& \frakb Y_1 - Y_2 +   \fraka Y_0, 
\eean
and thus 
$$
\frac12{d \over dt} Y_2 \le \frak C(1+\NN_F) (Y_1 + Y_0).
$$
Because the RHS term is a $L^1(0,T)$ function thanks  to  \eqref{eq:APB1} for $k = 1$, we deduce the announced control of the second moment. 

\smallskip\noindent
{\it Step~4. First moment of the flux.} 
We finally compute 
\bean
{d \over dt} \int_\OO F v \phi_1 
&=& - \int_{\Sigma_{12}} \gamma F J \nu  v \phi_1 + \int_\OO J F \phi'_1 v + \int_\OO (\frak K F - \fraka \partial_y F) \phi'_1 v
\\
&=&  - v_F \int_{\Sigma^+_2}     \gamma F  J \phi_1 + \int_\OO J F \phi'_1 v + \int_\OO (\frak b - y) F \phi'_1 v
 + \int_\OO   \fraka  F \phi''_1 v, 
\eean
using the Green formula and the boundary conditions several times. We deduce 
\bean
{d \over dt}  \int_\OO F v \phi_1 
\le      - v_F \int_{\Sigma^+_2}    \gamma F Jy  + v_E C_1 Z + (\fraka + \frakb )C_1 V, 
 \eean
 by throwing out  the negative terms, except for the boundary term. 
Using 
that $Z \le v_F Y_1$ and $(\fraka + \frakb ) V \le 2 \frak C (1+\NN_F) Y_0$ and the a priori estimate   \eqref{eq:APB1} for $k = 1$, we obtain
the announced $L^1$ control of the first moment of the flux through the boundary $\Sigma^+_2$. 
 \end{proof} 

 We next recall the entropy and Fisher information bound established in \cite[Theorem~6]{MR3177631}. Here and below, for $f : \OO \to \R_+$, 
 we denote by $\HH(f)$ the entropy and $\II(f)$ the (partial)  Fisher information defined by 
 $$
 \HH(f) := \int_\OO f \log f dydv, \quad 
 \II(f) := \int_\OO (\partial_y f)^2 f^{-1} dydv. 
 $$
 
  \begin{prop}\label{prop:APB2} 
 We assume 
 $$
 C_0 : = \int_\OO F_0 ( \langle y \rangle^2 + |\log F_0|) dvdy < \infty.
 $$
 For any $T >0$, there exists $C_T = C(T,C_0)$ such that 
\beqn\label{eq:APB2}
 \sup_{[0,T]} \int_\OO F |\log F| dydv + \int_0^T a_F \II(F) dt  +  \int_0^T\int_0^\infty   (J_+ \gamma_+ F) \log {J_- \over J_+}
  dydt  \le C_T.
\eeqn
 \end{prop}
 
 \begin{proof}[Proof of Proposition~\ref{prop:APB2}]
Denoting $H(F) = F \log F$, we first observe that 
\bean
&& \partial_t H(F) + \partial_v(JH(F)) + \partial_y ( (\frak K F  - \frak a \partial_y F) H'(F))   
\\
&&\qquad =      (\partial_vJ) (H(F) - F H'(F)) +  (\frak K F  - \frak a \partial_y F) \partial_y H'(F)
\\
&&\qquad =     -  (\partial_vJ) F +  \frak K \partial_y F  - \frak a (\partial_y F  )^2 F^{-1},  
\eean 
 so that 
 \bean
 {d \over dt} \int_\OO H(F) 
 &=&  - \int_{\Sigma_{12}} J \nu H(\gamma F)   -  \int_\OO (\partial_vJ) F + \int_\OO\frak K \partial_y F -  \frak a  \II(F) 
\\
 &=&  - \int_{\Sigma_{12}} J \nu H(\gamma F)   +  Y_1 + (1+y_L) Y_0 + \frak b \int_\OO  \partial_y F -  \frak a  \II(F), 
 \eean
 where we have used the Green formula and the boundary conditions  \eqref{eq:VCktBd1} and \eqref{eq:VCktBd1-linear}.
  For the boundary term, we observe that thanks to the boundary condition   \eqref{eq:VCktBd2}, we have 
\bean
 \int_{\Sigma_{12}} J \nu H(\gamma F) 
 &=& \int_{\Sigma_2^+} J \gamma F (\log J\gamma F - \log J) - \int_{\Sigma_2^-} J \gamma F (\log J\gamma F - \log J) 
\\
&=& \int_{y_F}^\infty J_+ \gamma_+ F (\log J_-  - \log J_+).
\eean
For the penultimate term, we write 
$$
 \frak b \int_\OO  \partial_y F \le  \frak b \eps \II(F) +  \frak b \eps^{-1} Y_0,
 $$
 thanks to the Young inequality. Choosing $\eps > 0$ small enough such that $ \frak b \eps \le  \fraka/2$, we deduce 
 \bean
 {d \over dt} \int_\OO H(F) +  \int_{y_F}^\infty J_+ \gamma_+ F \log {J_-  \over J_+} + \frac \fraka 2   \II(F)
 \le    (1+y_L) Y_0 +   \eps^{-1} \frak C (1+\NN_F), 
 \eean
 and from the  estimates yet established in Proposition~\ref{prop:APB1}, we conclude that 
 $$
 \int_\OO H(F(t,\cdot)) +  \int_0^t \int_{y_F}^\infty J_+ \gamma_+ F \log {J_-  \over J_+} +  \int_0^t\frac \fraka 2   \II(F)
\le C_T, 
$$
for any $t \in (0,T)$. On the other hand, from 
the elementary inequality 
$
s(\log s)_- \le \sqrt{s} {\bf 1}_{0 \le s \le e^{-y^2}} + s y^2  {\bf 1}_{e^{-y^2} < s \le 1} , 
$ for any $s,y > 0$, we classically deduce 
\bean
 \int_\OO F |\log F| 
 &=& \int_\OO H(F) + 2  \int_\OO F (\log F)_-
 \\
 &\le& \int_\OO H(F) + \sqrt{2   \pi} v_F + 2 Y_2. 
 \eean
We immediately  conclude to \eqref{eq:APB2} from the two last estimates together with the second moment estimate established in Proposition~\ref{prop:APB1}. 
\end{proof}

 We establish a last estimate on the trace function.  
 
  \begin{prop}\label{prop:APB3} 
  Under the assumptions of Proposition~\ref{prop:APB2}, for any $T >0$, there exist a function $\chi_0 \in C^2(\bar\UUU)$, such that $\chi_0 > 0$ on $\Gamma_+$, and a positive constant $C_{T}$ such that 
\beqn\label{eq:APB3}
  \int_0^T\int_0^\infty  \chi_0 J_+ \gamma_+ F \log \gamma_+ F
  dydt  \le C_{T}. 
\eeqn
 \end{prop}

\begin{proof}[Proof of Proposition~\ref{prop:APB3}] 
 With the notations of  Proposition~\ref{prop:APB2}, we observe that  
\bean
&& \partial_t H(F) + \partial_v(JH(F)) + \partial_y ( \frak K H(F)) -   \frak a \partial^2_{yy} H(F) + \frak a (\partial_{y} F)^2  H''(F)   
 \\
&&\qquad =   -  (\partial_vJ + \partial_y  \frak K ) F . 
\eean 
For any  $T > 0$ and $0 \le \chi_0 \in C^2([0,T] \times [0,v_F] \times [0,\infty))$ with supp$\chi_0 = [0,T] \times [v_F/2,v_F] \times [y_F,\infty)$,  we deduce 
\bean
&&\int_{\Gamma^+_2} J   \chi_0 H(\gamma F)   + \int_\UUU \fraka (\partial_y F)^2 H''(F) \chi_0 
=   - \int_\UUU (\partial_v J + \partial_y \frak K) F \chi_0 
\\
&&
\quad +\int_\UUU H(F) ( \partial_t \chi_0  + J \partial_v\chi_0 + \frak K  \partial_y \chi_0+  \fraka \partial_{yy}^2 \chi_0) . 
\eean
 We now fix $\chi_0$ such that $\chi_0 + |\partial_t \chi_0| + |\partial_v \chi_0| + |\partial_y \chi_0| + |\partial_{yy} \chi_0|  \le 1/y$ on $\UUU$, so that 
$$
|\partial_t \chi_0| + |J \partial_v \chi_0| +  \chi_0 | \partial_v J| + \chi_0 |\partial_y \frak K| \lesssim 1
$$
and 
$$
  |\frak K \partial_y \chi_0| +  \fraka | \partial_{yy}^2 \chi_0 |   \lesssim 1 + \NN_F. 
$$
We obtain 
\bean
\int_{\Gamma^+_2} J_+ \chi_0 H(\gamma_+ F)  
&\lesssim& 
  \int_\UUU F (1 + |\log F|) (1 + \NN_F) 
\\
&\lesssim& \sup_{[0,T]} \int_\OO F (1 + |\log F|)  \int_0^T (1 + \NN_F)  \le C_T, 
\eean
thanks to Proposition~\ref{prop:APB1} and Proposition~\ref{prop:APB2}. 
 \end{proof}

We present two important consequences of the above a priori estimates. We  first claim that both the solution and its trace on the outgoing boundary belong to weakly compact sets. 
That will be fundamental when we will deal with the stability and the existence results in the next sections. 

 \begin{cor}\label{cor:compactness}
  (1) The set  
 $$
 \CCC_\UUU := \Bigl\{ 0 \le F \in L^1(\UUU); \sup_{[0,T]} \int_\OO F ( \langle y \rangle  + |\log F|) dydv \le C_T \Bigr\} 
  $$
  is weakly compact in $L^1(\UUU,dydvdt)$. 
 
(2) The set  
 $$
 \CCC_\Gamma := \Bigl\{ 0 \le g \in L^1(\Gamma_{2}^{+});  
 \int_0^T \!\! \int_{y_F}^\infty J_+ g \Bigl(  \langle y \rangle  + \log {J_- \over J_+} +  \chi_0 \log g  \Bigr) dy dt
  \le C_T \Bigr\} 
  $$
  is weakly compact in $L^1(\Gamma_+,J_+dydt)$. 
 \end{cor}
 
 \begin{proof}[Proof of Corollary~\ref{cor:compactness}] The point (1) is very classical and its proof is thus skipped, see e.g. \cite{MR1014927}. We   just explain how to establish (2). 
We  have 
 $$
 \int_0^T\int_{y_F}^{y_F+\eps} J_+ g dydt \le \eta_1(\eps) C_T, \quad \eta_1(\eps) :=  \Bigl( \log {J_- (y_F) \over J_+(y_F +\eps)} \Bigr)^{-1} \to 0 \ \hbox{ as } \ \eps \to 0, 
 $$
 because $J_+ (y)  = (y-y_F) (v_E-v_F) \le J_+(y_F + \eps) \le J_-(y_F ) \le  J_-(y) = y v_E$ for any $y \in [y_F,y_F+\eps]$ and $\eps > 0$ small enough, and 
 $$
 \int_0^T\int_R^\infty  J_+ g dydt \le \eta_2(R)   C_T, \quad \eta_2 (R) := \langle R \rangle^{-1}  \to 0 \ \hbox{ as } \ R \to \infty. 
 $$
 On the other hand, we observe that for any $0 < \eps < 1 < y_F+1 < R < \infty$, there exists $\eta_{\eps,R} > 0$ such that 
$$
\chi_0 (t,v_F,y) \ge \eta_{\eps,R} \quad \forall \, t \in [0,T], \ \forall \, y \in [y_F+\eps,R], 
$$
and thus 
$$
\eta_{\eps,R}  \int_0^T\int_{y_F+\eps}^R    J_+ g \log g   dy dt  \le C_T.
$$
We conclude thanks to the Dunford-Pettis Lemma. 
 \end{proof}

The second consequence is that the equation \eqref{eq:VCk} makes sense for functions satisfying the above natural bounds. Multiplying  indeed a solution $F$ to  the equation \eqref{eq:VCk} by a test function 
 $\psi \in C^1_c(\bar \UUU)$ and integrating in the three variables, we formally obtain after integration by parts
\bear
\label{eq:defSol-AB}
\int_\UUU F [ \partial_t \psi + J  \partial_v  \psi + K_F  \partial_{y} \psi] +   \int_\OO F_0  \psi(0,\cdot) =  \int_\UUU a_F \partial_y F \partial_y \psi +  \int_{\Gamma_2} \nu J \gamma F \psi,  
\eear
where the contributions on the boundary sets $\Gamma_1$ and  $\Gamma_0$ have disappeared because of the boundary conditions \eqref{eq:VCktBd0} and \eqref{eq:VCktBd1}. 
 We will explain the meaning of the last (trace) term in the next section. Here we consider the other terms, and more precisely we explain why the two nonlinear terms are well defined when $F$
 satisfies the estimates \eqref{eq:APB1} and \eqref{eq:APB2}.
 
 For the third term in \eqref{eq:defSol-AB}, we write 
$$
\int_\UUU F  K_F  \partial_{y} \psi = \int_\UUU F  (y_* - y)  \partial_{y} \psi + \frak c \int_0^T \NN_F \Bigl( \int_\OO F \partial_y \psi \Bigr) dt, 
$$
where the last term is clearly the product of a $L^1(0,T)$ function and a $L^\infty(0,T)$ function, because of \eqref{eq:APB1} and the very definition \eqref{eq:defNNF}. 
For the fifth term in \eqref{eq:defSol-AB}, we similarly first write 
\beqn\label{eq:def-Psi}
\int_\UUU a_F \partial_y F \partial_y \psi =   \int_0^T a_F  \Psi dt, \quad \Psi :=  \int_\OO \partial_y F \partial_y \psi dydv, 
\eeqn
and we claim that $\Psi  \in L^2(0,T)$. We indeed have 
\bean
 \int_0^T    \Psi^2 dt 
&\le&  \int_0^T   \Bigl( \int_\OO |\partial_y F|  dydv\Bigr)^2   dt \|\partial_y \psi \|_{L^\infty}^2
\\
&\le& \sup_{[0,T]} \int_\OO F dydv  \int_\UUU   {|\partial_y F|^2 \over F}  dydvdt  \|\partial_y \psi \|_{L^\infty}^2 , 
\eean
where we have used the Cauchy-Schwarz inequality in the second line. The RHS term is finite because of the first bound in \eqref{eq:APB1} and the second bound in \eqref{eq:APB2}. 
Now, we write 
$$
a_F  \Psi = a_F \Psi {\bf 1}_{|\Psi| \le 1}   + a_F \Psi {\bf 1}_{|\Psi| > 1}. 
$$
The first term at the RHS is  the product of a $L^1(0,T)$ function and a $L^\infty(0,T)$ function, so that its integral is well-defined. For the second term, arguing as above, we have  
\bean
 \int_0^T  a_F |\Psi| {\bf 1}_{|\Psi| > 1} dt 
&\le&  \int_0^T  a_F \Psi^2  dt 
\\
&\le& \sup_{[0,T]} \int_\OO F dydv  \int_\UUU  a_F  {|\partial_y F|^2 \over F}   \|\partial_y \psi \|_{L^\infty}^2 , 
\eean
which is finite because again of  \eqref{eq:APB1} and  \eqref{eq:APB2}. As a consequence, the function $a_F \Psi {\bf 1}_{|\Psi| > 1}$ belongs to $L^1(0,T)$. 
That ends the proof of the fact that \eqref{eq:def-Psi} is meaningful when $F$ satisfies the natural bounds   \eqref{eq:APB1} and \eqref{eq:APB2}.

  \medskip
\section{The trace problem}

\label{sec:TracePb}

In this section we present some trace results adapted to our problem. These ones are in the spirit of the trace results established in \cite{MR2721875} (see also \cite{MR1765137,MR1776840}) and in \cite{CFS+SM} (see also \cite{CM-Landau**,CGMM**}). 
We start considering the trace problem for solutions to the equation 
\beqn\label{eq:trace1}
\partial_t g + J \partial_v g =  G_0 + \partial_y G_1 
\quad\text{in}\quad \UUU := (0,T) \times \OO.
\eeqn
For further discussion we define the boundary set $\Sigma_{12,R} :=   \{ 0 , v_F\} \times (R^{-1}, R)$ 
for $0 < R <1$ and similarly  $\Gamma_{12,R}$.  We define then the Lebesgue space $\Lloc^2(\Gamma_{12}; \xi dydt) := \bigcap_R L^2(\Gamma_{12,R}; \xi dydt) $  for a weight function $\xi : \Gamma_{12} \to \R_+$.

\begin{theo}\label{theo:trace1} Let us consider  $g \in L^\infty(\UUU)$ such that $\fraka^{1/2} \partial_y g \in L^2(\UUU)$, $a_* \le \fraka \in L^1(0,T)$,  which satisfies \eqref{eq:trace1} in the distributional sense
with 
$G_0, G_1 \in \Lloc^1(\bar\UUU)$ and $ \fraka^{-1/2} G_1   \in \Lloc^2(\bar\UUU)$. 
There exist then a function $\gamma g \in L^\infty(\Gamma_{12})$ and, for any $t \in [0,T]$,  a function $\gamma_t g \in L^\infty(\OO)$ which are the traces of $g$ on the boundary $\Gamma_{12}$ and on the sections $\{t \} \times \OO$, in the sense that the following  Green  renormalized formula
\bear\label{eq:traceL}
&&
  \int_{0}^t \!\! \int_{\Sigma_{12}} \beta(\gamma \, g) \, \varphi \, J \nu \, dyds
+ \left[ \int_\OO \beta(\gamma_s g)  \varphi(s,\cdot) \,  dvdy\right]_0^t
\\
\nonumber
&&= \int_{\UUU} \left\{  \beta(g) \, ( \partial_t  \varphi + \partial_v (J \varphi))  + G_0 \varphi \beta'(g)  - G_1 \partial_y(\varphi \beta'(g) )  \right\}  dydvdt
\eear
holds for any renormalizing function $\beta \in C^2(\R)$ and any test function  $\varphi \in  \DD(\UUU \cup \Gamma_{12})$. Furthermore $t \mapsto \gamma_t g \in C([0,T];\Lloc^1(\OO))$, $g(t,\cdot) = \gamma_t g$ a.e. on $\OO$, for a.e. $t \in (0,T)$   and $\gamma g \ge 0$ on $\Gamma_{12}$ if $g \ge 0$ on $\UUU$. 
\end{theo}

\noindent
Let us make several comments. 

\smallskip\noindent\ 
(i) In some sense, the result extends \cite[Theorem~2]{MR1765137} where in the last reference the RHS source term  in \eqref{eq:trace1} is assumed to be a $L^1$ function.  

\smallskip\noindent\ 
(ii) The formula \eqref{eq:traceL} is meaningful because of the assumptions made on $g$, in particular we have  $\beta(g), \beta'(g), \beta''(g) \in L^\infty(\UUU)$ and 
$$
|G_1 \partial_y \beta'(g) | = |G_1 \beta''(g) \partial_y g| \lesssim ( \fraka^{-1/2} G_1) ( \fraka^{1/2} \partial_y g)  \in \Lloc^1(\bar\UUU), 
$$ 
as a product of two  functions of $\Lloc^2(\bar\UUU)$. 

\smallskip\noindent\ 
(iii)  We will often write indifferently $g(t,\cdot) = \gamma_t g $. 

\smallskip\noindent\ 
(iv) The trace functions are defined as the limit of the restrictions to the subsets  $\Gamma_{12}$ and $\{t \} \times \OO$ of a sequence of  functions $g_\eps \in C(\bar\UUU)$ which will be defined by  $g_\eps := g \circledast \rho_\eps$ for a convenient convolution product $\circledast$ and a sequence of mollifier $(\rho_\eps)$, see \eqref{eq:def-gt} and  \eqref{eq:def-gammag} below.  

\smallskip\noindent\ 
(v) At least when $\frak a \equiv 1$, an alternative and equivalent way to formulate our trace result is to introduce the Banach space 
$$
\XXX := \{ g \in L^\infty (\UUU); \ \partial_y g \in L^2_{tvy}, \ \partial_t  g+ J \partial_v g \in L^1_{tvy} + L^2_{tv}H^{-1}_y \}
$$
and to say that $C_c (\bar\UUU) \subset \XXX$ is dense and that the restriction mapping $g \mapsto g_{|\Gamma_{12}}$ defined on $C_c (\bar\UUU)$ can be extended by continuity to any function $g \in \XXX$. 

\smallskip
\begin{proof}[Sketch of proof of Theorem~\ref{theo:trace1}]
As already mentioned, the proof is very similar to the ones of  \cite[Theorem~2]{MR1765137} and  \cite[Theorem~4.2]{MR2721875}. We refer to these papers for more details. 

\smallskip\noindent
{\sl Step 1. Some a priori estimates.} We fix $\nu \in W^{1,\infty}(0,v_F) $ such that $\nu (0) = -1$,  $\nu (v_F) = 1$.   We multiply the equation \eqref{eq:trace1} by $g J \nu \chi$, $0 \le \chi \in C_c^1((0,\infty))$, and after integration in all the variables, we (at least formally) obtain 
\bean
&&
  \int_{t_0}^{t_1} \!\! \int_{\Sigma_{12}} (\gamma \, g)^2 \,  J^2 \, \chi \, dyds
+ \left[ \int_\OO g(s,\cdot)^2  J \nu \chi \,  dvdy\right]_{t_0}^{t_1}
\\
\nonumber
&&= \int_{t_0}^{t_1} \!\! \int_\OO \left\{  g^2  \partial_v (J^2 \nu) \chi    + 2 G_0 g J \nu \chi  - 2 G_1 \partial_y(g J \nu \chi  )  \right\}  dydvdt, 
\eean
for any   $0 \le t_0 < t_1 \le T$, and in particular 
\bear \label{eq:trace-bound1}
&&
  \int_{t_0}^{t_1} \!\! \int_{\Sigma_{12}} (\gamma \, g)^2 \,  J^2 \, \chi \, dyds \lesssim \sum_{i=1,2} \int_{\OO_R} g(t_i,\cdot)^2   \,  dvdy 
\\
\nonumber
&&
 +  \int_{t_0}^{t_1} \!\! \int_{\OO_R} \left\{  g^2 + |G_0 g| + |G_1g| + |G_1 \partial_y g|   \right\}  dydvdt, 
\eear
with $\OO_R := (0,v_F) \times (1/R,R) \supset (0,v_F) \times \hbox{supp} \chi$. 
We now multiply the equation \eqref{eq:trace1} by $g \psi$, $ \psi \in C_c^1(\OO)$, $0 \le \psi \le 1$, and after integration in all the variables, we (at least formally) obtain 
\bean
 \left[ \int_\OO g(s,\cdot)^2  \psi dvdy\right]_{t_0}^{t_1}
 = \int_{t_0}^{t_1} \!\! \int_\OO \left\{  g^2  \partial_v (J \psi)    + 2 G_0 g  \psi  - 2 G_1 \partial_y(g \psi )  \right\}  dydvdt, 
\eean
with $t_0,t_1 \in [0,T]$, $t_0 \not=t_1$, and thus 
\bear \label{eq:trace-bound2}
\int_{K_0} g(t_1,\cdot)^2   dvdy 
&\lesssim& \int_{K_1} g(t_0,\cdot)^2   dvdy 
\\
\nonumber
&&\!\!\!\!\!\!+  \int_{T_0}^{T_1} \!\! \int_{K_1}
 \left\{  g^2 + |G_0 g| + |G_1g| + |G_1 \partial_y g|  \right\}  dydvdt, 
\eear
with $K_0 := \{ \psi = 1 \}$, $K_1 := \hbox{supp} \psi$, $T_0 := \min(t_0,t_1)$, $T_1 := \max(t_0,t_1)$. 

\smallskip\noindent
{\sl Step 2. Approximation and passing to the limit.} 
We define $g_\eps (t,\cdot) = g(t,\cdot) \circledast \rho_\eps$, where for $h \in \Lloc^1(\bar\OO)$, $h$ extended by $0$ to $\R^2 \backslash \OO$, and  $\rho_\eps(y,v) := \eps^{-2} \rho(y/\eps, v/\eps)$, $0 \le \rho \in C^2_c(\R^2)$, supp$\rho \subset [-1,1]^2$, 
we define the  modified convolution $h \circledast \rho_\eps$ by 
$$
h\circledast \rho_\eps(v,y) : =  \int_{\R^2} h(v-v_*-2\eps\nu(v),y-\eps-y_*) \rho_\eps(v_*,y_*) dv_*dy_*.
$$
We similarly define $G_{i,\eps} := G_i \circledast \rho_\eps$ and we 
denote $g_{\eps,\eps'} := g_\eps - g_{\eps'}$, $G_{i,\eps,\eps'} := G_{i,\eps} - G_{i,\eps'}$, $r_\eps := J \partial_v g_\eps - (J \partial_v g) \circledast \rho_\eps$ and $r_{\eps,\eps'} := r_\eps - r_{\eps'}$, so that 
\beqn\label{eq:trace2}
\partial_t g_{\eps} + J \partial_v g_{\eps} = r_{\eps} +  G_{0, \eps} + \partial_y G_{1, \eps}
\eeqn
and then 
\beqn\label{eq:trace3}
\partial_t g_{\eps,\eps'} + J \partial_v g_{\eps,\eps'} = r_{\eps,\eps'} +  G_{0, \eps,\eps'} + \partial_y G_{1, \eps,\eps'}, 
\eeqn
both in the distributional sense in $\UUU$. Thanks to classical convolution results, the   DiPerna-Lions commutator Lemma  \cite[Lemma~II.1]{MR1022305} and its extension {\it up to the boundary} \cite[Lemma~1]{MR1765137} (see also  \cite[Lemma~4.6]{MR2721875}), we have 
\beqn\label{eq:trace4}
g_\eps \to g, \quad  G_{i,\eps} \to G_i,  \quad  r_\eps \to 0 \quad \hbox{in} \quad \Lloc^1(\bar\UUU), 
\eeqn
with $(g_\eps)$   bounded in $L^\infty(\UUU)$,  $(r_\eps)$ bounded in $\Lloc^p(\bar\UUU)$ for any $p \in (1,\infty)$  and
\beqn\label{eq:trace5}
 \fraka^{1/2}  \partial_y g_\eps \to  \fraka^{1/2}  \partial_y g, \quad 
  \fraka^{-1/2}   G_{1,\eps} \to \fraka^{-1/2}   G_{1} \quad \hbox{in} \quad \Lloc^2(\bar\UUU),
\eeqn
from what we deduce 
\beqn\label{eq:trace6}
g_{\eps,\eps'}^2 + |(r_{\eps,\eps'} + G_{0,\eps,\eps'} ) g_{\eps,\eps'}| + |G_{1,\eps,\eps'} g_{\eps,\eps'}| + |G_{1,\eps,\eps'} \partial_y g_{\eps,\eps'}|   \to 0
 \quad \hbox{in} \quad \Lloc^1(\bar\UUU). 
\eeqn
On the other hand, thanks to the fact that $g_\eps$ is defined as a convolution product in $v$ and $y$ and to equation \eqref{eq:trace2}, we see that $g_\eps \in W^{1,1} (0,T;  \Wloc^{1,\infty}(\bar\OO)) \subset C(\bar\UUU)$, so that we may justify the  computations of Step~1 leading to the estimates \eqref{eq:trace-bound1} and \eqref{eq:trace-bound2}. More precisely, we may first fix a time $t_0 \in (0,T)$ such that 
 $g_\eps(t_0,\cdot) \to g(t_0,\cdot)$ in $\Lloc^p(\OO)$ for any $p \in [1,\infty)$ and, starting from \eqref{eq:trace3} and using \eqref{eq:trace6}, we get 
\bean
&&\sup_{t_1 \in [0,T]} \int_{K_0} g_{\eps,\eps'}(t_1,\cdot)^2   dvdy 
\lesssim \int_{K_1} g_{\eps,\eps'}(t_0,\cdot)^2   dvdy 
\\
\nonumber
&&+  \int_{0}^{T} \!\! \int_{K_1} \left\{ g_{\eps,\eps'}^2 + |(r_{\eps,\eps'} + G_{0,\eps,\eps'} ) g_{\eps,\eps'}| + |G_{1,\eps,\eps'} g_{\eps,\eps'}| + |G_{1,\eps,\eps'} \partial_y g_{\eps,\eps'}|   \right\}  dydvdt \to 0, 
\eean
as $\eps,\eps' \to 0$. 
We have established that $(g_\eps)$ is a Cauchy sequence in $C([0,T];\Lloc^2(\OO))$ and there thus exists  $\tilde g \in C([0,T];\Lloc^2(\OO))$ such that $g_\eps \to \tilde g$ in $C([0,T];\Lloc^2(\OO))$. 
Together with \eqref{eq:trace4} and the uniform bound on $(g_\eps)$, we deduce that 
$\tilde g = g$ a.e. (and from now on we use the same notation $g$ for both functions) and 
$g_\eps \to g$ in $C([0,T];\Lloc^p(\bar\OO))$ for any $p \in [1,\infty)$, in particular 
\beqn\label{eq:def-gt}
g(t,\cdot) = \gamma_t g := \lim_{\eps \to 0} g (t,\cdot) \circledast \rho_\eps \ \hbox{ in } \ \Lloc^p(\bar\OO), \quad \forall \, p \in [1,\infty), \ \forall \, t \in [0,T].
\eeqn
Using \eqref{eq:trace-bound1} for the solution $g_{\eps,\eps'}$ of \eqref{eq:trace3} and the already established convergence, we have 
\bean
&&
 \int_{\Gamma_{12,R}} (g_{\eps,\eps'|\Gamma_{12}})^2 \,  J^2 \,  dyds \lesssim \sup_{[0,T]} \int_{\OO_{R'}} g_{\eps,\eps'}(t,\cdot)^2   \,  dvdy 
\\
\nonumber
&&
 +  \int_{t_0}^{t_1} \!\! \int_{\OO_{R'}} \left\{ g_{\eps,\eps'}^2 + |(r_{\eps,\eps'} + G_{0,\eps,\eps'} ) g_{\eps,\eps'}| + |G_{1,\eps,\eps'} g_{\eps,\eps'}| + |G_{1,\eps,\eps'} \partial_y g_{\eps,\eps'}|   \right\}  dydvdt \to 0, 
\eean
 as $\eps,\eps' \to 0$ for any $0 < R < 1 < R'$.
 It is worth emphasizing that $g_{\eps |\Gamma_{12}}$ is well defined because $g_{\eps} \in C(\bar\UUU)$.  
 As a consequence, we have   established that $(g_{\eps|\Gamma_{12}})$ is a Cauchy sequence in $\Lloc^2(\Gamma_{12}; J^2 dydt)$ and there thus exists  $\gamma g \in \Lloc^2(\Gamma_{12}; J^2 dydt)$  such that $g_{\eps|\Gamma_{12}} \to \gamma g$ in $\Lloc^2(\Gamma_{12}; J^2 dydt)$. 
From the above convergences and bounds, we obtain  
\beqn\label{eq:def-gammag}
\gamma g :=   \lim_{\eps \to 0} (g  \circledast \rho_\eps)_{|\Gamma_{12}} \ \hbox{ in } \ \Lloc^p(\Gamma_{12};  dydt), \  \forall \, p \in [1,\infty), \  
\quad
\gamma g \in L^\infty(\Gamma_{12}).
\eeqn
In the case when $g \ge 0$ on $\UUU$, we have $g_\eps \ge 0$ on $\bar\UUU$ by construction, and thus $\gamma g \ge 0$ on $\Gamma_{12}$.
Coming back to \eqref{eq:trace2} and using the chain rule, we have 
\bean
&&
  \int_{0}^t \!\! \int_{\Sigma_{12}} \beta(g_{\eps|\Gamma_{12}}) \, \varphi \, J \nu \, dyds
+ \left[ \int_\OO (\beta(g_{\eps})  \varphi)(s,\cdot) \,  dvdy\right]_0^t
\\
&&= \int_{\UUU} \left\{  \beta(g_\eps) \, ( \partial_t  \varphi + \partial_v (J \varphi))  + (r_\eps+G_{0,\eps}) \varphi \beta'(g_\eps)  - G_{1,\eps} \partial_y(\varphi \beta'(g_\eps) )  \right\}  dydvdt,
\eean
for any renormalizing function $\beta \in C^2(\R)$ and any test function  $\varphi \in  \DD(\UUU \cup \Gamma_{12})$. 
We thus deduce \eqref{eq:traceL} by passing to the limit in the above identity. 
\end{proof}

 We now recall the theory of renormalized solutions due to DiPerna \& Lions for  kinetic Fokker-Planck equations first introduced in \cite{MR972541}. 
 We particularize our discussion to the case of  a measurable function $f:\UUU \to \R_+$ which satisfies 
\beqn\label{eq:trace-boundL1HF}
\sup_{[0,T]} \int_\OO f (1 + y + |\log f|) dydv + \int_\UUU \fraka |\partial_y \sqrt{f}|^2 dydv  < \infty 
\eeqn
as well as the linear  kinetic Fokker-Planck equation
\beqn\label{eq:kFP-lin}
 \partial_t f + \partial_v (Jf) + \partial_y( \frak K f) - \fraka \partial_{yy}^2 f  = 0 
\eeqn
in the distributional sense in $\UUU$,  
 with $\frak K = \frak b - y$, $\fraka, \frakb \in L^1(0,T)$, $\fraka \ge a_*$. It is worth emphasizing that $\frak K f, \frak a f \in L^1(\UUU)$ because of \eqref{eq:trace-boundL1HF}, so that \eqref{eq:kFP-lin} is meaningful. 
 We define   $\BB_1$ as the class of functions $\beta \in C^2(\R;\R)$ such that $0 \le -\beta''(s)  \lesssim (1+s)^{-1}$, so that $|\beta'(s)|   \lesssim 1 + \log (1+s)$ and $|\beta(s)|   \lesssim 1 + |s \log s|$.
 
\begin{theo}\label{theo:renormalization} Any measurable and nonnegative function $f$ on $\UUU$ satisfying   \eqref{eq:trace-boundL1HF}-\eqref{eq:kFP-lin} is also a renormalized solution in the sense that 
\beqn\label{eq:renormalizedSOL}
 \partial_t \beta(f) +  J \partial_v  \beta(f) +  \frak K  \partial_y\beta(f) - \fraka \partial_{yy}^2 \beta(f)   + f \beta'(f) (\partial_v J -1) +  \fraka \beta''(f) (\partial_{y} f)^2 = 0
 \eeqn
 in the distributional sense in $\UUU$ for any $\beta \in \BB_1$.  
\end{theo}

\begin{proof}[Sketch of the proof of Theorem~\ref{theo:renormalization}]
The proof is similar to the proof of  \cite[Theorem~5.2]{MR2721875}, see also \cite{MR972541},  and we thus just allude it. 
The equation is meaningful because all the terms $\beta(f)$,  $\frak K \beta(f)$,  $\fraka \beta(f)$, $f \beta'(f)$ and $\fraka \beta''(f) (\partial_{y} f)^2$ belong to  $\Lloc^1(\UUU)$ as a direct consequence of the bounds
\eqref{eq:trace-boundL1HF} on $f$ and the bounds on $\frak K$,  $\fraka$ and $\beta$. 

\smallskip
 Because it is a local property, we may localize in the $v$ and $y$ variables by introducing a set $\OO' \subset\subset \OO$, a function $0 \le \chi \le 1$, $\sqrt{\chi} \in W^{1,\infty}(\R^2)$, $\chi \equiv 1$ on $\OO'$, supp$\chi \subset\subset \OO$ and next defining $\bar f := f \chi$. Observing that 
 $$
 \int \bar f |\log \bar f| \le  \int \chi  f |\log f| +  \int f   \chi |\log \chi|
 $$
 and 
 $$
\int_{\UUU} \fraka(\partial_y \sqrt{\bar f})^2   \le 2 \int_{\UUU} \fraka   (\partial_y \sqrt{ f})^2 + 2 \int_{\UUU} \fraka  f  ( \partial_y \sqrt{\chi} )^2,
$$
we see that  $\bar f$ satisfies the same bound \eqref{eq:trace-boundL1HF} and the equation 
$$
 \partial_t \bar  f + J   \partial_v  \bar  f  +  \frak K     \partial_y  \bar  f  -   \fraka  \partial_{yy}^2  \bar  f  =  Q,
$$
 in the sense of distributions in $\VVV :=  (0,T) \times \R^2$, with
$$
Q := \bar f (1 - \partial_v J ) + f ( J   \partial_v \chi  +  \frak K     \partial_y  \chi) -  \fraka   \partial_y f    \partial_y \chi 
-  \fraka   f \partial^2_{yy} \chi \in L^1(\VVV), 
$$
and thus $Q  \in L^1(\VVV)$ and $Q =   f (1-\partial_v J)$ in $\UUU' := (0,T) \times \OO'$.   In the above computation, we have used that $\partial_y \frak K = -1$.
Because we are now working in the whole space, 
 
we may introduce the mollifier $(\rho_\eps)_{\eps > 0}$ with $\rho_\eps (v,y) := \xi_\eps(v) \xi_\eps(y)$ and $(\xi_\eps)$ is the gaussian function with variance $\eps$ and we may define  
$f_\eps := \bar f  * \rho_\eps$,  where $* = *_{v,y}$ stands for the convolution operator in the $v$ and $y$ variables. From the above construction, we have 
\beqn\label{eq:bound-HI-feps}
\HH (f_\eps (t,\cdot)) \le \HH (  \bar f (t,\cdot)), \quad 
\II ( f_\eps (t,\cdot)) \le \II (   \bar f (t,\cdot)), \quad \forall \, \eps > 0, \ \hbox{for a.e. } t \in (0,T), 
\eeqn
what is a consequence of the fact that both the entropy $\HH$ and the (partial) Fisher information $\II$ are decreasing along the flow of the heat equation in  $\R^2$. 
We deduce that $(f_\eps)$ satisfies the similar estimate as  \eqref{eq:trace-boundL1HF} in $\VVV$, uniformly in $\eps > 0$, and the equation
$$
 \partial_t f_\eps + J   \partial_v f_\eps +  \frak K     \partial_y f_\eps -   \fraka  \partial_{yy}^2 f_\eps =  Q_\eps + r_\eps,
 $$
 in the distributional sense in $\VVV$, with 
\bean
Q_\eps :=    -Q*\rho_\eps, \quad 
r_\eps := J   \partial_v f_\eps - (J \partial_v \bar f) * \rho_\eps  + \frak K     \partial_y f_\eps  - (\frak K \partial_y \bar f) * \rho_\eps. 
\eean
We have $f_\eps \in W^{1,1}(0,T;\Wloc^{2,\infty}(\R^2))$, from the definition of $f_\eps$ as a convolution and the above equation.
 The chain rule thus applies and we have 
\beqn\label{eq:betafeps}
 \partial_t \beta(f_\eps) + J   \partial_v \beta(f_\eps) +  \frak K     \partial_y \beta(f_\eps) -   \fraka  \partial_{yy}^2 \beta(f_\eps) =      \fraka  \beta''(f_\eps) (\partial_y f_\eps)^2  + (Q_\eps + r_\eps) \beta'(f_\eps) ,
\eeqn
 in the distributional sense in $\VVV$. We also observe that 
\beqn\label{eq:fepsTObarf}
 f_\eps \to \bar f \hbox{ strongly in } L^1(\VVV) \hbox{ and is uniformly bounded in } L^\infty(0,T;L^1(\R^2)), 
\eeqn
as well as  $Q_\eps \to Q$  in $L^1(\VVV)$ because $Q \in L^1(\VVV)$ 
 and $r_\eps \to 0$ in $L^1(\VVV)$ thanks to DiPerna-Lions commutator Lemma  \cite[Lemma~II.1]{MR1022305}. We recognize the assumptions of  \cite[Theorem~5.2]{MR2721875}, 
 except for the function $\fraka \in L^1(0,T)$ which was a positive constant in that previous framework.  We could adapt the same arguments. We rather follow a slightly different way that we believe to be simpler. Because the Fisher information $\II$ is lsc (for the weak convergence of measures), we have $ \II (   \bar f (t,\cdot)) \le \liminf  \II ( f_\eps (t,\cdot))$. Together with the second information in \eqref{eq:bound-HI-feps}, we have 
 $$
\| \partial_y \sqrt{f_\eps(t,\cdot)}  \|_{L^2}^2 = \frac14  \II ( f_\eps (t,\cdot)) \to   \frac14  \II (   \bar f (t,\cdot)) = \| \partial_y \sqrt{ \bar f (t,\cdot)}  \|_{L^2}^2, 
$$
as $\eps \to 0$, for any fixed $t \in (0,T)$. From \eqref{eq:fepsTObarf}, we know that $\partial_y \sqrt{f_\eps(t,\cdot)}  \wto \partial_y \sqrt{ \bar f (t,\cdot)} $ weakly in $\DD'(\R^2)$.
The two last convergences together imply that $\partial_y \sqrt{f_\eps(t,\cdot)}  \to \partial_y \sqrt{ \bar f (t,\cdot)} $ strongly  in $L^2(\R^2)$ as $\eps \to 0$, for any $t \in (0,T)$. 
Using the domination 
$$
\fraka  \II ( f_\eps (t,\cdot)) \le \fraka \II (   \bar f (t,\cdot)) \in L^1(0,T), 
$$
from \eqref{eq:bound-HI-feps} again and the fact that $\bar f$ satisfies the same bound \eqref{eq:trace-boundL1HF}, we may use the convergence dominated theorem of Lebesgue in order to obtain that 
$\fraka (\partial_y \sqrt{f_\eps})^2  \to \fraka ( \partial_y \sqrt{ \bar f  })^2 $ strongly  in $L^1((0,T) \times \R^2)$. We classically deduce that 
$$ 
 \fraka  \beta''(f_\eps) (\partial_y f_\eps)^2 =  4 \beta''(f_\eps) f_\eps  \fraka   (\partial_y \sqrt{f_\eps})^2 \  \to  \ 
4  \beta''(f) f  \fraka \bigl( \partial_y   {\sqrt{ \bar f }} \, \bigr)^2   =   \fraka  \beta''(f) (\partial_y \bar f)^2 
 $$
where we have used that $(\beta''(f_\eps) f_\eps)$ is bounded in $L^\infty$, $ \beta''(f_\eps) f_\eps \to  \beta''(f) f$ a.e. and the Egorov Theorem.  
Because we may easily pass to the limit in the other terms in \eqref{eq:betafeps}, we obtain \eqref{eq:renormalizedSOL} as the limit $\eps \to 0$ of  \eqref{eq:betafeps}. 
 Because $\OO' \subset\subset \OO$ is arbitrary, that ends the proof. \end{proof}

The two above results  together make possible to define a trace for solutions to kinetic Fokker-Planck equations in a non uniformly bounded framework.
 We define $\BB_2$ as the class of functions $\beta \in \BB_1$ such that furthermore $|\beta(s)| \lesssim \sqrt{s}$.

\begin{theo}\label{theo:trace2} Consider a measurable and nonnegative function $f$ on $\UUU$ satisfying   \eqref{eq:trace-boundL1HF}-\eqref{eq:kFP-lin}. Then 
$f \in C([0,T];L^1(\OO))$
and $f$ admits a trace $0 \le \gamma f \in \Lloc^1(\Gamma_{12}; J^2 dydt)$ in the sense that 
\bear\label{eq:trace-renorm}
&&
  \int_{0}^t \!\! \int_{\Sigma_{12}} \beta(\gamma \, f) \, \varphi \, J \nu \, dyds
+ \left[ \int_\OO \beta(f(s,\cdot))  \varphi(s,\cdot) \,  dvdy\right]_0^t
\\
\nonumber
&&= \int_{\UUU} \left\{  \beta(f) \, ( \partial_t  \varphi + \partial_v (J \varphi) + \partial_y (\frak K \varphi)) +  {\bf F}^\beta_0 \varphi   - {\bf F}^\beta_1 \partial_y\varphi   \right\}  dydvdt
\eear
for any $\beta \in \BB_2$ and $\varphi \in C^2_c(\UUU \cup \Gamma_{12})$ and for  any $\beta \in \BB_1$ and $\varphi \in C^2_c(\UUU \cup \Gamma_1 \cup \Gamma_2)$, where  
$$
 {\bf F}^\beta_0 :=   f \beta'(f) (1-\partial_v J) -  \fraka \beta''(f) (\partial_{y} f)^2, \quad
{\bf F}^\beta_1 :=   \fraka \partial_{y} \beta(f). 
$$
\end{theo}

A (usual) subtlety comes from the lack of control of  the trace function $\gamma f$ in the point $\{ (v,y) \in \Sigma_{12}; J(v,y) = 0 \} = \{(v_F,y_F)\}$ in order to give 
a sense to the first (boundary) integral in \eqref{eq:trace-renorm} when  $\beta \in \BB_1$ and $\varphi \in C^2_c(\UUU \cup \Gamma_{12})$. We however have $\beta(\gamma \, f)  J \in \Lloc^2(\Gamma_{12}; dydt)$
when $\beta \in \BB_2$ and $\varphi \in C^2_c(\UUU \cup \Gamma_{12})$ and   $\beta(\gamma \, f) J \in \Lloc^1(\Gamma_{12,R} \backslash \{ v_F \} \times [y_F-1/R,y_F+1/R]; dydt)$ when 
$\beta \in \BB_1$ and $\varphi \in C^2_c(\UUU \cup \Gamma_1 \cup \Gamma_2)$, supp$\, \varphi \cap \{ v_F \} \times [y_F-1/R,y_F+1/R] = \emptyset$, so that in both cases the  first (boundary) integral 
is well-defined. 

 \smallskip
\begin{proof}[Proof of Theorem~\ref{theo:trace2}] We argue similarly as in the proof of \cite[Theorem 5]{MR1776840} and \cite[Theorem~4.5]{MR2721875}. 
For  $\alpha(s) := s/(1+s)$,  Theorem~\ref{theo:renormalization} tells us that 
$$
  \partial_t g +  J \partial_v  g = G^\alpha_0 + \partial_y G^\alpha_1, 
 $$
 with 
 $$
  g:= \alpha(f), \quad
G^\alpha_0 := - \frak K  \partial_yg  -  f \alpha'(f) (\partial_v J -1) -   \fraka \alpha''(f) (\partial_{y} f)^2, \quad
G^\alpha_1 :=   \fraka \partial_{y} g. 
 $$
 It is worth emphasizing that $G^\alpha_0 \in  \Lloc^1(\bar\UUU)$, because in particular 
 $$
 |\partial_y\alpha(f)| \lesssim |\partial_y f| \lesssim f + { (\partial_y f)^2 \over f} \in L^1(\UUU)
 $$
 and 
 $$
 | \fraka \alpha''(f) (\partial_{y} f)^2| \lesssim  \fraka  { (\partial_y f)^2 \over f} \in L^1(\UUU).
 $$
Similarly, we have  $G^\alpha_1 \in L^1(\UUU)$, because 
 $$
 |\fraka \partial_y\alpha(f)| \lesssim \fraka f +  \fraka { (\partial_y f)^2 \over f} \in L^1(\UUU) 
 $$
and $\fraka^{-1/2} G^\alpha_1 \in L^2(\UUU)$, because 
 $$
  |\fraka^{-1/2} G^\alpha_1| = \fraka^{1/2} {|\partial_y f | \over (1+f)^2} \in L^2(\UUU).
  $$
 From Theorem~\ref{theo:trace1}, there exist some trace functions $0 \le \gamma^\alpha \in L^\infty(\Gamma_{12})$ and $\gamma^\alpha_t \in L^\infty(\OO)$ such that 
 \bean
&&
  \int_{0}^t \!\! \int_{\Sigma_{12}} \beta_0(\gamma^\alpha)  \, \varphi \, J \nu \, dyds
+ \left[ \int_\OO \beta_0(\gamma^\alpha_s)   \varphi(s,\cdot) \,  dvdy\right]_0^t
\\
\nonumber
&&= \int_{\UUU} \left\{  \beta_0(g) \, ( \partial_t  \varphi + \partial_v (J \varphi))  + G^\alpha_0 \varphi \beta_0'(g)  - G^\alpha_1 \partial_y(\varphi \beta_0'(g) )  \right\}  dydvdt,
\eean
for any $\varphi \in C^2_c(\UUU \cup \Gamma_{12})$ and $\beta_0 \in C^2([0,1])$. 
Defining $\gamma f := \alpha^{-1}(\gamma^\alpha)$ and $\gamma_t f := \alpha^{-1}(\gamma_t^\alpha)$ and using the chain rule in the RHS term, we deduce \eqref{eq:trace-renorm}
with $\beta := \beta_0 \circ \alpha$ and $\varphi \in C^2_c(\UUU \cup \Gamma_{12})$. More precisely,  defining $\BB$ the class of renormalizing functions $\beta \in  C^2(\R)$ such that $\beta' \in C^1_c(\R)$, we have established that \eqref{eq:trace-renorm} holds for any $\beta \in \BB$ by choosing $\beta_0 (\tau) := \beta(\alpha^{-1}(\tau))$ (so that supp$\, \beta_0' \subset [0,1)$ and we may define $\beta_0 \in C^2([0,1])$ by an extension by continuity). 

\smallskip
Choosing $\varphi := (J\nu) \chi$, $\chi \in C^2_c(\UUU \cup \Gamma_{12})$, $0 \le \chi \le 1$, supp$\, \chi \subset [0,T] \times \OO_R$, $R > 1$, and proceeding exactly as in the proof of Proposition~\ref{prop:APB3}  and of estimate \eqref{eq:trace-bound1}, we obtain 
\bear\label{eq:trace2-renorm-beta-bound}
&&
  \int_0^T \!\! \int_{\Sigma_{12}} \beta (\gamma \, f) \,  J^2 \, \chi \, dyds \lesssim \sup_{[0,T]}  \int_{\OO_R} \beta(f(t,\cdot)) \,  dvdy 
\\
\nonumber
&&
 +  \int_{t_0}^{t_1} \!\! \int_{\OO_R} \left\{ ( |\beta(f)| + |f \beta'(f)| ) (1 + |\frak b|) + \frak a |\beta''(f)| (\partial_y f)^2  \right\}  dydvdt. 
\eear
Choosing a sequence $(\beta_\eps)$ of $\BB$ such that $\beta_\eps \nearrow s \log (1+s)$, $|\beta'_\eps(s) | \le C (1 + \log (1+s))$ and $|\beta''_\eps(s) | \le C (1+ s)^{-1}$, 
we may pass to the limit in \eqref{eq:trace2-renorm-beta-bound} and we conclude to $\gamma f \log (1 + \gamma f) \in \Lloc^1(\Gamma_{12};J^2 dydt)$. On the other hand, using that $\beta(\gamma_t f) = \beta_0 (\gamma^\alpha_t) \in C([0,T];\Lloc^1(\OO))$ and \eqref{eq:trace-boundL1HF}, we deduce that $f  \in C([0,T];L^1(\OO))$. Using these estimates on $f$ and $\gamma f$, we may come back to \eqref{eq:trace-renorm} which already holds for $\beta \in \BB$ and $\varphi \in C^2_c(\UUU \cup \Gamma_{12})$, and then extend the formula to the two classes of renormalizing functions-test functions as in the statement by just using a density argument. 
 \end{proof}

  \begin{cor}\label{cor:trace2} For $f$ satisfying the conditions of Theorem~\ref{theo:trace2} so that its trace $\gamma f$ is well-defined  and under the additional assumption $\gamma f \in L^1(\Gamma_{12};|J| dy dt)$, the following Green formula  
\bear\label{eq:trace-non-renorm}
&&
  \int_{0}^t \!\! \int_{\Sigma_{12}} \gamma \, f \, \varphi \, J \nu \, dyds
+ \left[ \int_\OO f(s,\cdot)  \varphi(s,\cdot) \,  dvdy\right]_0^t
\\
\nonumber
&&= \int_{\UUU} \left\{  f  \, ( \partial_t  \varphi + J \partial_v   \varphi + \frak K \partial_y  \varphi) -   \fraka \partial_{y} f \partial_y\varphi   \right\}  dydvdt
\eear
holds for any  $\varphi \in C^1_c(  \UUU \cup \Gamma_{12})$. 
 \end{cor}

  \begin{proof}[Proof of Corollary~\ref{cor:trace2}] Defining $\beta_\eps(s) := s/(1+\eps s)$ for $\eps > 0$, we observe that $\beta_\eps  \in \BB$,  $\beta_\eps'' \le 0$,   and we may thus write \eqref{eq:trace-renorm} with this choice of renormalizing functions. Using $0 \le \beta_\eps(s) \nearrow s$, $0 \le \beta'_\eps(s) \nearrow 1$, $| \beta''_\eps(s) | \le \eps/s \to 0$ and the estimates satisfied by $f$ and $\gamma f$, we may pass to the limit as $\eps \to 0$ and we obtain \eqref{eq:trace-non-renorm}. 
  \end{proof}

 \medskip
\section{Definition of solution and stability result} 
\label{sec:def&stab}
From the material of the preceding sections, we are now able to formulate the notion of solutions we will deal with. 

\begin{definition}\label{def:weak-sol}
We say that $F $ is a global weak solution to   the Voltage-Conductance kinetic (VCk) equation   \eqref{eq:VCk},  \eqref{eq:defNNF}, \eqref{eq:VCkt=0}, \eqref{eq:VCkBd} 
if $0 \le F \in C([0,\infty);L^1(\OO))$ satisfies the estimates \eqref{eq:APB1} with $k=2$, \eqref{eq:APB2}, \eqref{eq:APB3}, as well as 
\bear
\label{eq:defSol}
\int_\UUU F [ \partial_t \psi + J  \partial_v  \psi + K_F  \partial_{y} \psi] +   \int_\OO F_0  \psi(0,\cdot) =  \int_\UUU a_F \partial_y F \partial_y \psi +  \int_{\Gamma_{12}} \gamma F  \psi J \nu, 
\eear
for any $\psi \in C^1_c(\bar \UUU)$, where the trace function $ \gamma F \in L^1(\Gamma_{12}; |J| dtdy)$ is defined thanks to  Theorem~\ref{theo:trace2} and satisfies \eqref{eq:VCktBd1} and \eqref{eq:VCktBd2} pointwise. 
\end{definition} 
 
Let us make a few comments about this definition. It has already been discussed at the end of Section~\ref{sec:A priori estim} the fact that the interior terms are well defined when the estimates 
\eqref{eq:APB1}, \eqref{eq:APB2}, \eqref{eq:APB3} hold. 
Because of the trace estimate in \eqref{eq:APB1}, we have $\fraka := a_* +  \cc^2 \NN(\gamma_+ F) \in L^1(0,T)$  and $\frakb := y_* + \cc \NN(\gamma_+ F) \in L^1(0,T)$, and  
we may apply the trace result formulated in  Corollary~\ref{cor:trace2}  in order to link the function $F$ and its trace function $\gamma F$ on $\Gamma_{12}$ through the Green formula \eqref{eq:trace-non-renorm}. 
It is worth emphasizing that 

the  boundary condition  \eqref{eq:VCktBd0} is encapsulated in the set of test function $\psi \in C^1_c(\bar \UUU)$, just as it is the case for the usual formulation of Neumann condition for elliptic equations in a not necessary smooth domain.

\smallskip
The above definition of solution is suitable for our purpose as we see now by stating and proving the corresponding weak stability result, in the spirit of the DiPerna-Lions theory for kinetic equations.

\begin{theo}\label{theo:stability} 
Let the connectivity parameter $\cc \ge 0$ arbitrarily fixed and let us consider a sequence  $(F_n)$ of solutions to the Voltage-Conductance kinetic (VCk) equation   \eqref{eq:VCk},  \eqref{eq:defNNF}, \eqref{eq:VCkt=0}, \eqref{eq:VCktBd0}, \eqref{eq:VCktBd1},  \eqref{eq:VCktBd2} which satisfies uniformly (in $n$) the natural bound \eqref{eq:APB1}, \eqref{eq:APB2}, \eqref{eq:APB3}.
Up to the extraction of a subsequence, the sequence  $(F_n)$ converges to a function $F$ and this one  satisfies the same natural bound and is a solution to the Voltage-Conductance kinetic (VCk) equation
in the sense of Definition~\ref{def:weak-sol}.
\end{theo}

During  the proof, we will need a key auxiliary convergence result that we establish first. We introduce the truncation function $T_R$ defined by 
$$
T_R(s) := - R \hbox{ if } s  < -R, \  T_R(s) := s \hbox{ if } s \in [-R,R], \ T_R(s) := R \hbox{ if } s > R, 
$$
and the function $T^c_R(s) := s - T_R(s)$. 

\begin{lem}\label{lem:AB2} Consider two sequences $(A_n)$ and $(B_n)$ on $L^1(0,T)$ such that  
\bean
A_n \wto A \hbox{ weakly  in } L^1(0,T),  \quad B_n \to B  \hbox{ strongly  in } L^1(0,T)
, \quad  \int_0^T |A_n| B_n^2 \le C.
\eean
Then 
\beqn\label{eq:lemaB2}
A_n B_n \wto AB  \hbox{ weakly  in } L^1(0,T).
\eeqn
\end{lem}

\begin{rem}\label{rem:AB2} 
This result is a small  improvement of the  result claiming that \eqref{eq:lemaB2} holds when 
\bean
A_n \wto A \hbox{ weakly  in } L^1(0,T),  \quad B_n \to B  \hbox{ a.e. in } (0,T)
, \quad  (B_n)  \hbox{ bounded  in } L^\infty(0,T), 
\eean
which is a classical consequence of the Egorov Lemma, and thus also a direct consequence of  Lemma~\ref{lem:AB2}. 
\end{rem}

\begin{proof}[Proof of Lemma~\ref{lem:AB2}]
We fix $\psi \in L^\infty(0,T)$ and we split
 $$
 \int_0^T A_n B_n \psi   =  \int_0^T A_n T_R(B_n) \psi +  \int_0^T A_n T^c_R(B_n) \psi.
 $$
 For the first term, using the convergence result recalled in Remark~\ref{rem:AB2}, 
 we straightforwardly have 
\beqn\label{lem:AB-cvgce1}
 \int_0^T A_n T_R(B_n) \psi \to  \int_0^T A T_R(B) \psi. 
 \eeqn 
 From the reverse sense of the Dunford-Pettis Lemma, the sequence $(A_n)$ is bounded in $L^1$ and uniformly equi-integrable, so does is the sequence
$(|A_n|)$. Thanks to the Dunford-Pettis Lemma, we thus deduce that there exists $A^* \in L^1(0,T)$ such that, up to the extraction of a subsequence,  $ |A_n| \wto A^*$  weakly  in $L^1(0,T)$, 
so that $A^* \ge |A|$ (what we find by testing the sequence $(A_n)$ with the functions $\phi \, \hbox{sign}\, A$, $0 \le \phi \in L^\infty(0,T)$). 
Similarly as for \eqref{lem:AB-cvgce1}, we have 
$$
  \int_0^T |A_n| T_R(B^2_n) \to   \int_0^T A^* T_R(B^2),  
 $$
 from what we deduce 
$$
 \int_0^T |A| B^2  \le \lim_{R \to \infty}  \int_0^T A^* T_R(B^2)  \le  C. 
$$
On the other hand, for $(A',B') = (A_n,B_n)$ or $(A',B') = (A,B)$, we have  
  \beqn\label{lem:AB-cvgce2}
 \int_0^T A' T^c_R(B')  \psi \le  \int_0^T |A'| |B'| {\bf 1}_{|B'| \ge R} \| \psi \|_{L^\infty} \le {C \over R} \| \psi \|_{L^\infty}. 
\eeqn
  Using the same splitting  
   $$
 \int_0^T A B \psi   =  \int_0^T A T_R(B) \psi +  \int_0^T A T^c_R(B) \psi
 $$
 and the two pieces of information \eqref{lem:AB-cvgce1} and \eqref{lem:AB-cvgce2}, we immediately conclude. 
%
%
%
%
\end{proof}

 \begin{proof}[Proof of Theorem~\ref{theo:stability}]

 We consider a sequence of functions $(F_n)$ which satisfies the estimates \eqref{eq:APB1}, \eqref{eq:APB2}, \eqref{eq:APB3} uniformly in $n$ and such that $F_n$ is a solution to the VCk equation
in the sense of Definition~\ref{def:weak-sol}, that is 
\beqn
\label{eq:defSolFn}
\int_\UUU F_n [ \partial_t \psi + J  \partial_v  \psi + \frak K_{\NN_n}   \partial_{y} \psi] +   \int_\OO F_0  \psi(0,\cdot) =  \int_\UUU \fraka_{\NN_n}  \partial_y F_n \partial_y \psi +  \int_{\Gamma_{12}}  (\gamma F_n) J\nu \psi, 
\eeqn
for any $\psi \in C^1_c(\bar \UUU)$ and the trace function $ \gamma F_n$ satisfies \eqref{eq:VCktBd1} and \eqref{eq:VCktBd2} pointwise, that is 
\beqn\label{eq:bdycond-Fn}
 \gamma_1 F_n = 0  \ \hbox{ on } \ \Gamma_1, \quad J_- \gamma_- F_n = J_+ \gamma_+ F_n \ \hbox{ on } \ \Gamma_2^-. 
\eeqn
For a given function $0 \le \frak N \in L^1(0,T)$, we define 

 $$
\frak K_{\frak N} :=  y_* + \cc\, \frak N - y, 
\quad 
\fraka_{\frak N} :=  a_* + \cc^2 \frak N , 
\qquad
\NN_n := \NN (\gamma_+ F_n), 
$$
so that  $\frak K _{\NN_n} = K_{F_n}$ and  $\fraka_{\NN_n} = a_{F_n}$.

\smallskip
We split the proof into four  steps. 

\smallskip
\smallskip\noindent
{\sl Step 1. Convergences.}

\smallskip\smallskip\noindent
$\bullet$  
 Thanks to the Dunford-Pettis Lemma as formulated in Corollary~\ref{cor:compactness}-(1), up to the extraction of a subsequence, we have 
\bear\label{eq:theo-stab-cvgce1}
 F_n \wto F  \quad \hbox{weakly in } L^1(\UUU; dydvdt), 
\eear
for a function $0 \le F \in L^1(\UUU)$. The estimates \eqref{eq:APB1}, \eqref{eq:APB2}, \eqref{eq:APB3} and the  Corollary~\ref{cor:compactness}-(2), imply that 
the sequence $(\gamma_+ F_n)$ is weakly compact in $L^1(\Gamma_{2}^+; |J| dydt)$. Together with the boundary conditions  \eqref{eq:bdycond-Fn}, we deduce that 
the sequence $(\gamma F_n)$ is weakly compact in $L^1(\Gamma_{12}; |J| dydt)$. There thus exists $0 \le \bar\gamma \in L^1(\Gamma_{12}; |J| dydt)$ such that, up to the extraction of a subsequence, 
\beqn
\label{eq:theo-stab-cvgce2}
  \gamma F_n \wto \bar\gamma  \quad \hbox{weakly in } L^1(\Gamma_{12}; |J| dydt). 
\eeqn
Together with   \eqref{eq:bdycond-Fn} and the very definition \eqref{eq:defNNF} of $\NN$, we deduce that 
\bear\label{eq:theo-stab-cvgce12}
&&\bar\gamma_1 = 0   \  \hbox{on} \ \Gamma_{1}, 
\quad J_- \bar\gamma_- = J_+ \bar\gamma_+  \  \hbox{on} \ \Gamma_{2}^-, 
 \\
 &&\label{eq:CVGCE-Nn}
 \NN_n   \wto  \NN(\bar\gamma_+) =:   \bar \NN  \quad \hbox{weakly in } L^1(0,T),
 \eear
 where we have set $\bar\gamma_1 := \bar\gamma {\bf 1}_{\Gamma_1}$, $\bar\gamma_\pm := \bar\gamma {\bf 1}_{\Gamma_2^\pm}$.

 \Black
 
 \smallskip\smallskip\noindent
$\bullet$ We fix $\phi \in C_c^3(\UUU)$   and we denote 
\beqn\label{eq:defPhin}
\Phi_n := \int_\OO  \phi  \, \partial_y F_n. 
\eeqn
Observing that, by performing one integration  by part, 
$$
\Phi_n = - \int_\OO  F_n  \, \partial_y \phi ,
$$
we see that the sequence $(\Phi_n)$ is clearly bounded in $L^\infty(0,T)$ from the mass conservation. 
On the other hand, we compute 
\bean
{d \over dt} \Phi_n 
&=& -  {d \over dt}  \int_\OO  F_n  \partial_{y} \phi  
\\
&=& -  \int_\OO   \partial_t F_n  \partial_{y} \phi  -  \int_\OO F_n  \partial^2_{ty} \phi 
\\
&=& 
\int_\OO   F_n  [ J \partial^2_{vy} \phi - \frak K_{\NN_n} \partial^2_{yy} \phi - \fraka_{\NN_n} \partial^3_{yyy} \phi - \partial^2_{ty} \phi]  , 
\eean
where the RHS is bounded in $L^1(0,T)$ because it is the case for $(\NN_n)$ and because of the mass conservation again.
From the Rellich theorem, we deduce that $(\Phi_n)$ is relatively compact in $L^1(0,T)$, and from the fact that $\partial_y F_n \wto \partial_y F$ in $\DD'(\UUU)$, we finally conclude to
\beqn\label{eq:compactnessStrongMean}
\int_\OO   \phi  \, \partial_y F_n \to \int_\OO   \phi  \, \partial_y F
\ \hbox{ strongly } \ L^1(0,T).
\eeqn

 \smallskip\smallskip\noindent
$\bullet$ We now  assume $ \phi \in C_c(\bar\UUU)$ and we still define $\Phi_n$ thanks to \eqref{eq:defPhin}.
We may  introduce an approximation family $(\phi_\eps)$ of $C^3_c(\UUU)$ such that 
$\|\phi_\eps \|_{L^\infty} \le \|\phi \|_{L^\infty} $ 
  and 
  $\|\phi - \phi_\eps \|_{L^2}      \le \eps$, for any $\eps > 0$. 
With obvious notations, we have already established that the associated sequence $(\Phi_{n,\eps})_{n \ge 1}$ is relatively compact in $L^1(0,T)$ for any $\eps > 0$. On the other hand, we compute 
\bean
\| \Phi_{n,\eps} - \Phi_n \|_{L^1(0,T)} 
&\le& \int_\UUU |\partial_y F_n (\phi - \phi_\eps) | 
\\
&\le& \Bigl( \int_\UUU {(\partial_y F_n)^2 \over F_n} \Bigr)^{1/2} \Bigl( \int_\UUU F_n (\phi - \phi_\eps)^2  \Bigr)^{1/2} 
\\
&\le&  \Bigr( \int_0^T \II(F_n) \Bigr)^{1/2} \Bigl( M \|\phi - \phi_\eps \|_{L^2}^2  +  4 { \| \phi\|_{L^\infty}^2 \over \log M} \int_\UUU F_n (\log F_n)_+ \Bigr)^{1/2}, 
\eean
for any $n,M \ge 1$ and $\eps >0$, where we have used the Cauchy-Schwarz inequality in the second line and the classical estimate 
$$
s = s \wedge M + (s-M)_+ \le M + {1 \over \log M} s (\log s)_+, \quad \forall \, s \ge 0, \ M \ge 1, 
$$
in the last line. 
 Because of the uniform in $n$ Fisher information bound  \eqref{eq:APB2} (recall that $\fraka_{\NN_n} \ge a_* > 0$) and
the uniform in $n$ entropy bound \eqref{eq:APB2},
that  precisely means that the sequence $(\Phi_n)$ is precompact in $L^1(0,T)$, from what we classically deduce \eqref{eq:compactnessStrongMean} again. 

 \smallskip\smallskip\noindent
$\bullet$ For any $\phi \in C_c(\bar \UUU)$, using an approximation step by compact supported functions, we also prove similarly as above that 
\beqn\label{eq:compactnessStrongMeanBIS} 
\int_\OO  F_n  \phi \to \int_\OO F \phi
\ \hbox{ strongly } \ L^1(0,T), \quad \int_\OO  F_n  \phi  \ \hbox{ bounded in  } \ L^\infty(0,T).
\eeqn

\smallskip\noindent
{\sl Step 2. Passing to the limit in the equation.} For $\psi \in C^1_c(\bar \UUU)$, we may write 
$$
\int_\UUU F_n  \frak K_{\NN_n}  \partial_{y} \psi = \int_\UUU F_n (y_*  - y) \partial_y \psi +  \int_0^T  \cc \NN_n \Psi_n, 
\quad \Psi_n := \int_\OO F_n  \partial_y \psi. 
$$
From now on, we assume $\cc  > 0$, the proof in the linear case $\cc=0$ being simpler. 
Using \eqref{eq:theo-stab-cvgce1}, we may pass to the limit in the first term in the above splitting. 
Using  \eqref{eq:CVGCE-Nn}, \eqref{eq:compactnessStrongMeanBIS} with $\phi := \partial_y \psi$ and 
Remark~\ref{rem:AB2} with $A_n :=  \cc \NN_n$ and $B_n := \Psi_n$, we may pass to the limit  in the second term in the above splitting. Altogether, we obtain 
$$
\int_\UUU F_n  \frak K_{\NN_n}   \partial_{y} \psi  \to 
\int_\UUU F  \frak K_{\bar \NN}  \partial_{y} \psi.
$$
We may also write 
$$
\int_\UUU \fraka_{\NN_n}  \partial_y F_n \partial_y \psi= \int_0^T  A_n B_n, 
\quad
A_n := \cc^2 \NN_n, 
\quad B_n := \int_\OO \partial_y F_n \partial_y \psi. 
$$
Thanks to the Cauchy-Schwarz inequality, we notice that 
$$
\Bigl( \int_\OO \partial_y F_n \partial_y \psi \Bigr)^2  \le  \int_\OO {|\partial_y F_n|^2 \over F_n} \int_\OO  F_n | \partial_y\psi|^2, 
$$
so that 
\beqn\label{eq:borneAB2}
\int_0^T A_n B^2_n \le \| \partial_y \psi \|_{L^\infty}^2 \| F_0 \|_{L^1}   \int_\UUU a_{F_n}  {|\partial_y F_n|^2 \over F_n} \le C.
\eeqn
Using \eqref{eq:CVGCE-Nn}, \eqref{eq:compactnessStrongMean} with $\phi := \partial_y \psi$, \eqref{eq:borneAB2} and Lemma~\ref{lem:AB2},  
 we obtain again $$
\int_\UUU \fraka_{\NN_n}  \partial_y F_n \partial_y \psi \to \int_\UUU  \fraka_{\bar\NN}  \partial_y F \partial_y \psi.
$$
We now easily pass to the limit in \eqref{eq:defSolFn}, and for any $\psi \in C^1_c(\bar \UUU)$, we get
\bear
\label{eq:defSolBIS}
\int_\UUU F [ \partial_t \psi + J  \partial_v  \psi +  \frak K_{\bar\NN}  \partial_{y} \psi] +   \int_\OO F_0  \psi(0,\cdot) =  \int_\UUU \fraka_{\bar\NN}    \partial_y F \partial_y \psi +  \int_{\Gamma_{12}} J \nu  \bar\gamma  \psi. 
\eear

\smallskip\noindent
{\sl Step 3. Passing to the limit in the estimates.} First, from the estimates  \eqref{eq:APB1}, \eqref{eq:APB2} uniformly satisfied by $(F_n)$ and the convergences \eqref{eq:theo-stab-cvgce1}, \eqref{eq:theo-stab-cvgce2}, we immediately deduce that 
\beqn\label{eq:theo-stab-step3-1}
 \sup_{[0,T]} \int_\OO F \langle y \rangle^k dydv + \int_0^T\int_0^\infty  (J_+ \bar\gamma_+)  \langle y \rangle^{k-1} dy dt \le C_T
\eeqn
and 
\beqn\label{eq:theo-stab-step3-2}
 \sup_{[0,T]} \int_\OO F (\log F)_+ dydv +  \int_0^T\int_0^\infty   (J_+ \bar\gamma_+ ) \log {J_- \over J_+}
  dydt  \le C_T, 
\eeqn
for a constant $C_T = C_T(F_0)$, where we have just used that the above quantities are weakly lsc. We recover the $F(\log F)_-$ part for completing \eqref{eq:theo-stab-step3-2}
by using the same trick as at the end of the proof of Proposition~\ref{prop:APB2}.
We now establish 
\beqn\label{eq:theo-stab-step3-3}
\int_0^T \fraka_{\bar \NN}  \II(F) dt \le C_T.
\eeqn
We classically know, see for instance \cite[Lemma~3.5]{MR3188710}, that 
$$
\II (g) = \sup_{\psi \in C_c^2(\OO;\R)} \II_\psi (g), \quad \II_\psi (g) :=  \int_\OO (- \frac14 \psi^2 - \partial_y \psi) g,
$$
and we know from \eqref{eq:compactnessStrongMeanBIS} that 
$$
\II_\psi(F_n) \to \II_\psi(F) \quad\hbox{in}\quad L^1(0,T). 
$$
 We thus deduce 
\bean
 \int_0^T \fraka_{\bar \NN} \, ( \II_\psi(F) \wedge \ell)_+  
  &=&  \lim  \int_0^T \fraka_{ \NN_n}  ( \II_\psi(F_n) \wedge \ell)_+ 
\\
  &\le&  \liminf  \int_0^T \fraka_{ \NN_n}  \II(F_n) \le C_T, 
  \eean
for any $\ell \ge 0$ and $\psi \in C_c^2(\OO;\R)$, from what \eqref{eq:theo-stab-step3-3} follows by taking the supremum on $\psi$ and letting $\ell \to \infty$.

\smallskip\noindent
{\sl Step 4. Conclusion.} Because $F$ satisfies the estimates \eqref{eq:theo-stab-step3-1}, \eqref{eq:theo-stab-step3-2}, \eqref{eq:theo-stab-step3-3} and the equation \eqref{eq:defSolBIS}, which thus in particular holds in the distributional sense, we may apply Theorem~\ref{theo:trace2} which tells us that $F \in C([0,T];L^1(\OO))$ and there exists $\gamma F \in \Lloc^1(\Gamma_{12};J^2 dydt)$ such that 
\bear\label{eq:theo-stab-step4}
&&
  \int_{0}^t \!\! \int_{\Sigma_{12}} \gamma F \, \varphi \, J \nu \, dyds
+ \left[ \int_\OO F(s,\cdot)  \varphi(s,\cdot) \,  dvdy\right]_0^t
\\
\nonumber
&&= \int_{\UUU} \left\{  F  \, ( \partial_t  \varphi + J \partial_v   \varphi + \frak K_{\bar\NN}  \partial_y  \varphi) - \fraka_{\bar\NN}  \partial_{y} F \partial_y\varphi   \right\}  dydvdt
\eear
holds for any  $\varphi \in C^1_c(  \UUU \cup \Gamma_1 \cup \Gamma_2)$.  Comparing with \eqref{eq:defSolBIS}, we deduce that $\gamma F = \bar \gamma$ on $\Gamma_{12}$. 
From \eqref{eq:CVGCE-Nn}, we deduce that $\bar\NN = \NN(\gamma_+ F)$, so that $ \frak K_{\bar\NN} = K_F$, $\fraka_{\bar\NN}  = a_F$, and thus coming back to  \eqref{eq:defSolBIS}, we see that $F$ satisfies \eqref{eq:defSol}. 
From \eqref{eq:theo-stab-cvgce12}, we deduce that $F$ satisfies the boundary conditions \eqref{eq:VCktBd1} and  \eqref{eq:VCktBd2}. 
From \eqref{eq:theo-stab-step3-1}, \eqref{eq:theo-stab-step3-2} and \eqref{eq:theo-stab-step3-3}, we deduce that $F$  satisfies the estimates   \eqref{eq:APB1}, \eqref{eq:APB2} and \eqref{eq:APB3}.
  We have thus  $\gamma F \in L^1(\Gamma_{12},|J|dydt)$ and from  Corollary~\ref{cor:trace2}, equation \eqref{eq:theo-stab-step4} holds for any $\varphi \in   C^1_c(\bar \UUU)$.  
  \end{proof}

 \medskip
\section{Proof of Theorem~\ref{theo-Exists}} 
 \label{sec:existence}
 The whole section is dedicated to the existence Theorem~\ref{theo-Exists}, each subsection corresponding to a different step. 
 
\subsection{About the linear problem} 
\label{sec:LinearPb}

We recall and make a bit more precise some material developed in \cite{CFS+SM}. 
 We consider the linear  VCk equation
 \beqn\label{eq:VCklinear}
\partial_t f  = \LLL_{\fraka,\frak K} f := -  \partial_v(Jf) -  \partial_y(  \frak K f) +  \fraka  \partial_{yy}^2 f
\quad\text{in}\quad (0,\infty) \times \OO, 
\eeqn
where, for some $a^* > \max(a_*,y_*)$, 
\bean
 \frak K :=  \frakb - y, \quad  \fraka, \frakb  \in L^\infty(0,\infty), \quad  y_* \le \frakb \le a^*, \quad a_* \le  \fraka \le a^*, 
\eean 
and the evolution equation is complemented with the boundary conditions \eqref{eq:VCktBd1}, \eqref{eq:VCktBd2} and \eqref{eq:VCktBd1-linear}, 
that we will sometime  summarize   with the shorthand 
\beqn\label{eq:LinearVCkBd}
\RRR_{\fraka, \frak K} \gamma f = 0 \quad\text{on}\quad \Gamma.
 \eeqn
 
 Let us introduce some notations. For a given polynomial weight function $\omega(y) := (1+y)^k$, $ k > 0$, we define the weighted Lebesgue space $L^p_k$ associated to the norm
$$
\| f \|_{L^p_k} := \| f \omega \|_{L^p}.
$$
In the sequel, we will choose $k > 5/2$ in such a way that $L^2_k \subset L^1_2$. 
We also  denote $d\xi^\omega_1 := \omega^2  |J| dtdy$  and $d\xi^\omega_2 := \omega^2  J^2 /\langle y \rangle^2 dtdy$  the Borel measures on the boundary $\Gamma_{12}$.  We define $\BB_3$ as the class of functions $\beta \in C^2(\R)$ such that $\beta'' \in L^\infty(\R)$ and $\BB_4 \subset \BB_3$ as the class of functions $\beta \in C^2(\R)$ such that $\beta' \in W^{1,\infty}(\R)$.
  
 \smallskip
We start by stating a trace result adapted to the above linear framework. 

\begin{theo}\label{theo:trace3}
For any polynomial weight function $\omega(y) := (1+y)^k$,  $k > 5/2$, and any solution $f \in L^\infty(0,T;L^2_k) \cap L^2((0,T) \times (0,v_F); H^1(\R_+))$, $\forall \, T > 0$,  to the linear  VCk equation \eqref{eq:VCklinear}, \eqref{eq:LinearVCkBd}, there exist a trace function $\gamma f \in L^2(\Gamma_{12}; d\xi^\omega_2)$  and a family of trace functions $\gamma_t f \in L^2_\omega(\OO)  $ such that
\bear
\label{eq:theo:existL2-1}
&&  \int_\OO \beta(\gamma_T  f ) \psi (T,\cdot)  + \int_{\Gamma} J  \beta(\gamma f)  \psi  + \int_\UUU \beta(f) [ \partial_t \psi +  \partial_v (J \psi)]  - \int_\UUU  (\partial_vJ)f \beta'(f) \psi 
\\ \nonumber
&&\quad = -  \int_\UUU  (\frak K f - \fraka \partial_{y} f ) \partial_y  (\beta'(f)\psi)   +  \int_\OO \beta(\gamma_0 f ) \psi (0, \cdot) , 
\eear
 for any $\beta \in \BB_3$ and $\psi \in C^1_c(\UUU \cup \Gamma_{1} \cup \Gamma_2)$ and for any $\beta \in \BB_4$ and  $\psi \in C^1_c(\UUU \cup \Gamma_{12})$.  
Furthermore, $\gamma_t f = f(t,\cdot)$ a.e. on $(0,T)$ and $t \mapsto \gamma_t f \in C([0,T];\Lloc^2(\OO))$. \end{theo} 
 
 \begin{proof}[Proof of Theorem~\ref{theo:trace3}]
 That is a straightforward consequence of Theorem~\ref{theo:trace1} and Theorem~\ref{theo:renormalization}, observing that, because of the better integrability properties of the functions $f$ and $\partial_y f$ in the present framework, we can choose $\beta(s) := s^2 \in \BB_3$ with $\psi := \omega^2$ or $\psi := \omega^2 \langle y \rangle^{-2} J \nu$ 
  in \eqref{eq:theo:existL2-1}, what provide the convenient estimates on the trace functions $\gamma f $ and  $\gamma_t f$, so that we may repeat the proof of Theorem~\ref{theo:trace2}. For further reference, we  emphasize that the last choice of $(\beta,\psi)$ implies that 
\beqn
\label{eq:theo:existL2-1-bound}
   \int_{\Gamma}   (\gamma f)^2 J^2 \omega^2 \langle y \rangle^{-2} \le  C \Bigl \{ \sup_{[0,T]} \int_\OO f^2 \omega^2 + \int_\UUU (\partial_y f)^2 \omega^2 \Bigr\} ,
\eeqn
with $C = C_T (1 + \| \fraka \|_{L^\infty}^2 + \| \frakb \|_{L^\infty}^2)$. 
   \end{proof}

We now recall a 

\begin{theo}\label{theo:existL2}
For any polynomial weight function $\omega(y) := (1+y)^k$,  $k > 5/2$,
and any initial datum $f_0 \in L^2_k$, there exists a unique solution $f \in C([0,T];L^2_k) \cap L^2((0,T) \times (0,v_F); H^1(\R_+))$, $\forall \, T > 0$,  to the linear  VCk equation \eqref{eq:VCklinear}-\eqref{eq:LinearVCkBd} 
 and this one satisfies the growth estimate  
 \beqn\label{eq:lem:GrowthL2}
\sup_{t \in [0,T]} \| f_t \|^2_{L^2_k} + \int_0^T \| \partial_y f_t \|^2_{L^2_k}dt \le C e^{\kappa T } \| f_0 \|^2_{L^2_k}, \quad \forall \, T \ge 0,
\eeqn
for some constants $C = C_\omega \ge 1$ and  $\kappa \le C( 1+ \| \fraka \|_{L^\infty} + \| \frakb \|_{L^\infty})$.
More precisely, $f$ satisfies   \eqref{eq:theo:existL2-1} and the associated trace function $\gamma f$ satisfies the   boundary conditions \eqref{eq:VCktBd1}, \eqref{eq:VCktBd2} in the a.e. sense. 
The additional  boundary condition \eqref{eq:VCktBd1-linear} is encapsulated in the fact that $\psi$ does not necessarily  vanish on the boundary set $\Gamma_0$. 
Furthermore, if $f_0 \ge 0$, the solution $f$ satisfies 
$$
f(t,\cdot) \ge 0  \quad\hbox{and}\quad \| f(t,\cdot) \|_{L^1} = \| f_0 \|_{L^1}, \quad \forall \, t \ge 0. 
$$
\end{theo}

 \begin{proof}[About the proof of Theorem~\ref{theo:existL2}]
 That is nothing but \cite[Theorem~2.2]{CFS+SM} and the energy estimate \eqref{eq:lem:GrowthL2} is a straightforward consequence of  \cite[(2.8)]{CFS+SM} and of the Gronwall lemma or it is also the limit estimate in 
 \cite[(2.22)]{CFS+SM}. 
  \end{proof}

\Black

\subsection{A truncated problem} 
\label{sec:truncated-pb}
In this section, for $n \ge 1$,  
we consider the regularized Voltage-Conductance kinetic equation
\beqn\label{eq:VCkn}
\partial_t F  + \partial_v(JF) + \partial_y(K^n_F F) - a^n_F \partial_{yy}^2 F = 0
\quad\text{in}\quad (0,\infty) \times \OO, 
\eeqn
 where   the coefficients are given by 
\beqn\label{eq:VCkn-coeff}
  K^n_F :=  y_* + \cc\, \NN^n_F - y, 
\quad
a^n_F := a_* + \cc^2 \NN^n_F , 
\quad 
 \NN^n_F :=  \NN_F  \wedge  n, 
\eeqn
complemented with the boundary conditions \eqref{eq:VCktBd1}, \eqref{eq:VCktBd2} and 
\beqn\label{eq:VCktBd0-nonlinear-n}
  K^n_F \gamma F-  a^n_F  \partial_y \gamma F = 0 \ \hbox{ on } \ \Gamma_0, 
\eeqn
and with an initial datum
\beqn\label{eq:VCkTrunct=0}
F(0,\cdot) = F^n_0\quad\text{in}\quad \OO,
 \eeqn
 with $0 \le F^n_0 \in L^2_k(\OO)$, $k > 5/2$, so that $ L^2_k(\OO) \subset L^1_2(\OO)$. 
We aim to establish the following existence result. 
  
 \begin{prop}\label{prop:existence-Schauder} Under the above conditions, there exists at least one solution $F \in X_T := C([0,T]; L^2_k(\OO)) \cap L^2((0,T) \times (0,v_F); H^1(\R_+))$, $\forall \, T > 0$,  
 to the truncated Voltage-Conductance kinetic equation \eqref{eq:VCkn},  \eqref{eq:VCktBd1}, \eqref{eq:VCktBd2},  \eqref{eq:VCktBd0-nonlinear-n}, \eqref{eq:VCkTrunct=0}. 
 \end{prop}

The rest of this section will be devoted to clarifying the meaning of the statement and proving this result. For that purpose we will use the Schauder  fixed point theorem. That is a small variant with respect to the proof of   \cite[Theorem~1.1]{CFS+SM} which is based on the Tykonov fixed point theorem. The issue comes from the way we regularize the equation thanks to the truncation operation $ \NN_F  \wedge  n$ which  makes a little more complicated   the proof of Proposition~\ref{prop:existence-Schauder}    than it would be if we had used another regularization operation (as a convolution for instance)  
but this has the advantage of making possible to use in the next step (see  Section~\ref{subsec:passingTOlimit}) the a priori estimates established in Section~\ref{sec:A priori estim}.
We split the proof of Proposition~\ref{prop:existence-Schauder} into four steps. We first define a suitable function $\Upsilon$ associated to our problem. In a second step, we prove that this mapping enjoys a continuity property and, in a third step, that it enjoys a strong compactness property. In a last step, we use the Schauder  fixed point theorem in order to conclude.

\smallskip

\smallskip
$\bullet$
For a given $\frak N \in L^2(0,T)$, we define $f = f_{\frak N}$ the solution to the linear equation 
\beqn\label{eq:VCkTrunc-Lin}
\partial_t f  + \partial_v(Jf) + \partial_y(\frak K_{\frak N^n} f) - \frak a_{\frak N^n} \partial_{yy}^2 f = 0, 
\quad\text{in}\quad (0,\infty) \times \OO,
\eeqn
with  
\beqn\label{eq:def-Nn} 
\frak K_{\frak N^n} :=  \frak b_{\frak N^n} - y, 
\quad
\frak b_{\frak N^n} := y_* + \cc\, \frak N^n,  
\quad
\frak a_{\frak N^n} := a_* + \cc^2 \frak N^n, 
\quad
\frak N^n := \frak N \wedge n,
\eeqn
complemented with the boundary condition \eqref{eq:LinearVCkBd} associated to $\frak K  := \frak K_{\frak N^n} $ and $\fraka :=\fraka_{\frak N^n}$ and with
the initial datum \eqref{eq:VCkTrunct=0}. 
 From Theorem~\ref{theo:existL2},  we know that there exists a unique solution $f \in X_T$ to equation \eqref{eq:VCkTrunc-Lin},  \eqref{eq:LinearVCkBd} 
associated to the initial datum \eqref{eq:VCkTrunct=0} and from  Theorem~\ref{theo:trace3}, we know that $f$ has  a trace $\gamma f$ which satisfies 
$$
J\gamma f \in L^2((0,T) \times \R_+; \langle y \rangle^{k-1} dydt).
$$
Let us make more precise the available estimates. 

\smallskip
On the one hand, we have 
\bean
\int_0^T \NN_f \, dt 
&=&  \int_0^T  \int_0^\infty J_+ \gamma f dy   dt 
\\
&\le&  \Bigl( \int_0^T  \int_0^\infty (J_+ \gamma f)^2  \omega^2 \langle y \rangle^{-2} dy  dt \Bigr)^{1/2} \Bigl( T\int_0^\infty \omega^{-2} \langle y \rangle^{2} dy \Bigr)^{1/2}, 
\eean
so that 
$$
\| \NN_f \|_{L^1(0,T)} \le C_{\omega,T}  \|  \gamma f    \|_{L^2(d\xi^\omega_2)}.
$$
From \eqref{eq:theo:existL2-1-bound} and \eqref{eq:lem:GrowthL2}, we also have 
\beqn
\label{eq:theo:regularizedPB-estimgammaf}
    \|  \gamma f    \|_{L^2(d\xi^\omega_2)} \le  A \| F_0 \|_{L^2_k}, 
 \eeqn
 with $A = A(T,n)$. Defining $R :=  C_{\omega,T} A  \| F_0 \|_{L^2_k}$ and 
 $$
\NNN_R := \{ \frak N \in L^1(0,T); \ \| \frak N \|_{L^1} \le R \}, 
$$
we thus have 
$$
\Upsilon : \NNN_R \to \NNN_R, \quad \frak N \mapsto \Upsilon(\frak N) := \NN_f.
$$

\smallskip
On the other hand, defining 
 $\frak C := \max( y_* + \frak c n, a_* + \frak c^2 n)$, we have  
 $$
  \frakb_{\frak N^n} \le \frak C (1+\NN_f),  \quad  a_*   \le  \fraka_{\frak N^n} \le \frak C (1+\NN_f). 
 $$
 We may then justify (with the help of Theorem~\ref{theo:trace3}) the formal computations of Propositions~\ref{prop:APB1}, ~\ref{prop:APB2}, ~\ref{prop:APB3}
 and we deduce, with the notation of  Corollary~\ref{cor:compactness},  that $f \in  \CCC_\UUU$ and $\gamma f \in  \CCC_\Gamma$ with a constant $C_T = C(T,n)$ independent of  $\frak N$.

\smallskip
$\bullet$   We  claim that the mapping $\Upsilon$ is  continuous from $L^1(0,T)$ endowed with the strong topology into $L^1(0,T)$ endowed with the weak topology. 
We thus consider a sequence $( \frak N_\ell)$ such that $ \frak N_\ell \to  \frak N$ strongly in $L^1(0,T)$ and we prove that   $\Upsilon (\frak N_\ell) \wto \Upsilon (\frak N)$ weakly in $L^1(0,T)$.  
On the one hand, the sequence $f_\ell := f_{\frak N_\ell}$
is bounded in $X_T$ because of \eqref{eq:lem:GrowthL2}, 
 from what we deduce for a subsequence $(f_{\ell'})$ and a function $f \in L^\infty(L^2_\omega) \cap L^2_{tv}H^1_y$ that 
\beqn\label{eq:regularizedPB-cvgce-interior}
f_{\ell'} \wto f \hbox{ weakly in } L^\infty(L^2_\omega) \cap L^2_{tv}H^1_y.
\eeqn
Because $\frak N^n_{\ell'}   \to \frak N^n $ strongly in $L^2(0,T)$, we easily pass to the limit in the weak formulation of the equation 
\eqref{eq:VCkTrunc-Lin},   \eqref{eq:LinearVCkBd}, \eqref{eq:VCkTrunct=0} satisfied by $(f_{\ell'},\frak N_{\ell'})$ and we obtain that $(f,\frak N)$ satisfies the same equation in the distributional sense.  
More precisely, observing that $(\gamma f_{\ell'})$ satisfies \eqref{eq:theo:regularizedPB-estimgammaf} uniformly in $\ell'$, there exists a function $\bar \gamma \in L^1((0,T) \times (0,\infty); |J| dy dt)$ such that  for a subsequence $(\gamma f_{\ell''})$ of $(\gamma f_{\ell'})$ there holds
\beqn\label{eq:regularizedPB-cvgce-boundary}
\gamma f_{\ell''} \wto \bar \gamma \,\,\,  \hbox{ weakly in } \,\,\, L^1((0,T) \times (0,\infty); |J| dy dt), 
\eeqn
and passing to the limit in the equation 
$$ \int_{\Gamma_{12}} \nu J   \gamma f_{\ell''} \psi + \int_\UUU f_{\ell''} [ \partial_t \psi +  J \partial_v  \psi]  
+ \int_\UUU  (\frak K_{\frak N^n_{\ell''}} f_{\ell''} - \fraka_{\frak N^n_{\ell''}} \partial_{y} f_{\ell''} ) \partial_y  \psi  =  \int_\OO F_0^n \psi (0, \cdot) , 
$$
we get 
\beqn\label{eq:regularizedPB-limiteq}  
\int_{\Gamma_{12}} \nu J   \bar\gamma  \psi + \int_\UUU f [ \partial_t \psi +  J \partial_v  \psi]  
+ \int_\UUU  (\frak K_{\frak N^n} f - \fraka_{\frak N^n} \partial_{y} f ) \partial_y  \psi  =  \int_\OO F_0^n \psi (0, \cdot) , 
\eeqn
for any $\psi \in C^1_c(\bar \UUU)$. Passing to the limit in the boundary conditions  \eqref{eq:VCktBd1} and \eqref{eq:VCktBd2} satisfied by $\gamma f_{\ell''}$, we also deduce that 
$\bar\gamma$ satisfies \eqref{eq:theo-stab-cvgce12}. 
On the other hand, from   \eqref{eq:regularizedPB-limiteq} and the trace Theorem~\ref{theo:trace3}, we have $\bar \gamma = \gamma f$ a.e. on $\Gamma_{12}$, so that $f$ satisfies the boundary conditions  \eqref{eq:VCktBd1} and \eqref{eq:VCktBd2}. Since the last boundary condition  \eqref{eq:VCktBd0-nonlinear-n} is directly encapsulated in the weak formulation \eqref{eq:regularizedPB-limiteq}, we have established that $f$ is the   
unique solution   to equation \eqref{eq:VCkTrunc-Lin},  \eqref{eq:LinearVCkBd},  \eqref{eq:VCkTrunct=0} associated to $\frak N$. 
Together with \eqref{eq:regularizedPB-cvgce-boundary} and similarly as for \eqref{eq:CVGCE-Nn}, we also have 
\bean
 \Upsilon (\frak N) = \NN(\gamma_+ f) = \lim  \NN(\gamma_+ f_{\ell''})   =  \lim \Upsilon (\frak N_{\ell''}), 
\eean
and by uniqueness of the limit, we get that $\Upsilon $ is continuous.

\smallskip
$\bullet$ We establish now that $\Upsilon$ is compact. We thus consider a sequence $(\frak N_\ell)$ of  $\NNN_R$ and we have to prove that we may find a subsequence of $(\Upsilon(\frak N_\ell))$ which strongly converges in $L^1(0,T)$. With the above notations, the associated sequence $(f_\ell)$ is bounded in $X_T \subset \CCC_\UUU$. On the other hand, there exists $ \bar {\frak N}  \in L^\infty(0,T)$ 
such that, up to the extraction of a subsequence,  
\beqn\label{eq:TruncatedEq-Nell-barN}
\frak N_{\ell}^n  =  (\frak N_{\ell} ) \wedge n \wto \bar {\frak N}  \ \hbox{ weakly in }\  L^\infty(0,T). 
\eeqn
We may thus repeat  the proof of Theorem~\ref{theo:stability} or the  above proof of the continuity property of $\Upsilon$ and we get the interior convergence \eqref{eq:regularizedPB-cvgce-interior} with $f$ satisfying 
\beqn\label{eq:TruncatedEq-on-f}
 \partial_t  f +    \partial_v (Jf)  +  \partial_y(\frak K_{\bar {\frak N}}  f) -     \fraka_{\bar {\frak N}} \partial_{yy}^2 f   = 0,  
\eeqn
 as well as, up to the extraction of a subsequence,  the weak convergence at the boundary 
 \beqn\label{eq:TruncatedEq-cvgce-boundary}
\gamma f_{\ell} \wto  \gamma f \,\,\,  \hbox{ weakly in } \,\,\, L^1(\Gamma_{12}; |J| dy dt). 
\eeqn  
In order to improve that last convergence, we follow \cite[Proof of Theorem 5.2]{MR2721875} (see also \cite{MR972541}). 
We just give the idea, the only novelty being the non constant factor in front of the diffusion term and a small simplification of the argument.  

\smallskip
We take $\beta(s) := \log (1+s)$, so that $\beta \in \BB_2$, and we define  $g_\ell := \beta(f_\ell)$, so that 
\beqn\label{eq:log-f-ell}
 \partial_t  g_\ell +  J \partial_v g_\ell  +  \frak K_{\frak N_{\ell}^n} \partial_y g_\ell   - \fraka_{\frak N_{\ell}^n} \partial_{yy}^2 g_\ell =   f_\ell \beta'(f_\ell) (1-\partial_v J) 
 +  \fraka_{\frak N_{\ell}^n}  (\partial_{y} g_\ell)^2, 
\eeqn
 thanks to Theorem~\ref{theo:renormalization}. 
From \cite[Theorem~4.1]{CFS+SM}, which is a consequence of \cite[Theorem~1.3]{zbMATH07050183}, 
the sequence $(f_\ell)$ belongs to a strong compact set of $L^1(\UUU)$, so that, up to the extraction of a subsequence,  
\beqn\label{eq:TruncatedEq-fn-f}
f_\ell \to f \ \hbox{ strongly in } \ L^1(\UUU). 
\eeqn
As a consequence, we have $g_\ell  \to g := \beta(f)$ strongly in $L^2(\UUU)$. We claim that, up to the extraction of a subsequence, 
\beqn\label{eq:TruncatedEq-a1}
 \fraka_{\frak N_{\ell}^n}  (\partial_{y} g_\ell)^2  \wto \fraka_{\bar{\frak N}  }    (\partial_{y} g)^2 + \xi \ \hbox{ weakly in } \ \DD'(\UUU), 
\eeqn
for a bounded measure $\xi \ge 0$. 
In order to see that, we first observe that the sequence $( \fraka_{\frak N_{\ell}^n}  (\partial_{y} g_\ell)^2)$ is bounded in $L^1(\UUU)$, so that there exists a bounded measure $\Xi \ge 0$ such that 
$$
 \fraka_{\frak N_{\ell}^n}  (\partial_{y} g_\ell)^2  \wto \Xi \ \hbox{ weakly in } \ \DD'(\UUU).
$$
On the other hand, we observe that 
\beqn\label{eq:TruncatedEq-a2}
  \fraka_{\frak N_{\ell}^n}  (\partial_{y} g_\ell)^2 \ge  \fraka_{\frak N_{\ell}^n}  [ \phi^2 - 2 \phi  \partial_{y} g_\ell]
\eeqn
  for any $\phi \in C_c(\bar\UUU)$. 
Repeating the proof of \eqref{eq:compactnessStrongMean} in the present context, we see that 
\beqn\label{eq:TruncatedEq-strongcon-moyenne}
\int_\OO   \phi  \, \partial_y g_\ell  \to \int_\OO   \phi  \, \partial_y g
\ \hbox{ strongly } \ L^1(0,T), 
\eeqn
 so that we may pass to the limit in \eqref{eq:TruncatedEq-a2}, and we get 
  $$
\Xi \ge  \fraka_{\bar{\frak N}}  [ \phi^2 - 2 \phi  \partial_{y} g]
  $$
  for any $\phi \in C_c(\bar\UUU)$. We conclude to \eqref{eq:TruncatedEq-a1} by maximizing in $\phi$. 
 We finally observe that because $(\gamma f_{\ell})$ satisfies \eqref{eq:theo:regularizedPB-estimgammaf} uniformly in $\ell$, the same holds for $(\gamma g_{\ell})$ and there thus exists a function $\bar\beta \in L^2(\Gamma_{12},d\xi^\omega_2)$, such that, up to the extraction of a subsequence, 
 \beqn  \label{eq:TruncatedEq-gamma-gell-tobeta}
 \gamma g_{\ell} \wto \bar\beta \ \hbox{ weakly in } \ L^2(\Gamma_{12},d\xi^\omega_2).
 \eeqn
We now want to pass to the limit in   equation \eqref{eq:log-f-ell} up to the outgoing boundary. More precisely, from equation \eqref{eq:log-f-ell}  and  Theorem~\ref{theo:trace2}, we may  write
 \bean
&&
 \int_{\Gamma_{2}^+}  \gamma g_{\ell}  \, \varphi \, J \, dyds
= \int_{\UUU} \left\{  g_\ell \, ( \partial_t  \varphi + \partial_v (J \varphi) + \partial_y ( \frak K_{\frak N_{\ell}^n}  \varphi)) +   {\bf F}^\beta_{0\ell}\varphi   - {\bf F}^\beta_{1\ell} \partial_y\varphi   \right\}  dydvdt, 
\eean
for any $\varphi \in C^2_c(\UUU \cup \Gamma_2^+)$, with 
 $$
 {\bf F}^\beta_{0\ell} :=    f_\ell \beta'(f_\ell) (1-\partial_v J ) +   \fraka_{\frak N_{\ell}^n}  (\partial_{y} g_\ell)^2,
 \quad 
 {\bf F}^\beta_{1\ell} :=   \fraka_{\frak N_{\ell}^n}   \partial_{y} g_\ell, 
 $$
 and where we have observed that  $-\beta''(f_\ell) \fraka_{\frak N_{\ell}^n}   (\partial_{y} f_\ell)^2 =  \fraka_{\frak N_{\ell}^n}  (\partial_{y} g_\ell)^2$. 
Using \eqref{eq:TruncatedEq-Nell-barN}, \eqref{eq:TruncatedEq-fn-f}, \eqref{eq:TruncatedEq-a1}, \eqref{eq:TruncatedEq-strongcon-moyenne} and \eqref{eq:TruncatedEq-gamma-gell-tobeta}, we may pass to the limit  in the above equation, and we get 
 \bean
 \int_{\Gamma_{2}^+} \bar\beta \, \varphi \, J \, dyds
= \int_{\UUU} \left\{  \beta(f) \, ( \partial_t  \varphi + \partial_v (J \varphi) + \partial_y (\frak K_{\bar{\frak N}}  \varphi)) +   ({\bf F}^\beta_0 + \xi)\varphi   - {\bf F}^\beta_1 \partial_y\varphi   \right\}  dydvdt,
\eean
with 
$$
{\bf F}^\beta_0 :=   f \beta'(f) (1-\partial_v J )  +   \fraka_{\bar{\frak N}}  (\partial_{y} g)^2, \quad
{\bf F}^\beta_1 :=   \fraka_{\bar{\frak N}}  \partial_{y} g.
$$
On the other hand, because $f$ satisfies \eqref{eq:TruncatedEq-on-f} and thanks to the trace Theorem~\ref{theo:trace2}, we have  
 \bean
&&
 \int_{\Gamma_{2}^+} \beta(\gamma_+ f) \, \varphi \, J \nu \, dyds
= \int_{\UUU} \left\{  \beta(f) \, ( \partial_t  \varphi + \partial_v (J \varphi) + \partial_y (\frak K \varphi)) +   {\bf F}^\beta_0 \varphi   - {\bf F}^\beta_1 \partial_y\varphi   \right\}  dydvdt, 
\eean
observing again that  $- \beta''(f)  \fraka_{\bar{\frak N}}  (\partial_{y} f)^2 =  \fraka_{\bar{\frak N}}  (\partial_{y} g)^2 $. 
Comparing the two equations, we deduce that 
 \bean
  \int_{\Gamma^+_{2}} \bar\beta \, \varphi \, J  \, dyds -   \int_{\UUU}  \xi \varphi  dydvdt =    \int_{\Gamma^+_{2}}  \beta(\gamma_+ \, f) \, \varphi \, J   \, dyds,
\eean
for any $ 0 \le \varphi \in C^2_c(\UUU \cup \Gamma_2^+)$, or equivalently  
$$
\liminf \beta(\gamma_+ f_{\ell}) = \bar\beta \ge \beta(\gamma_+ f) \  \hbox{ on } \ \Gamma_2^+.
$$
Since $\beta$ is a concave function, the reverse inequality $ \beta(\gamma_+ f) \ge \limsup \beta(\gamma_+ f_{\ell})$ is true, and we thus deduce that 
$$
\beta(\gamma_+ f_{\ell}) \wto \beta(\gamma_+ f) \  \hbox{ weakly on } \ \Gamma_2^+.
$$
Together with \eqref{eq:TruncatedEq-cvgce-boundary} 
and a classical convexity argument (see for instance \cite{MR1260437}), we deduce the 
$$
\gamma_+ f_{\ell} \to  \gamma_+ f \  \hbox{ strongly in }  \ L^1((0,T) \times \Sigma_2^+;Jdtdy),  
$$
from what we immediately get
$$
\Upsilon (\frak N_{\ell}) = \NN(\gamma_+ f_{\ell}) \to \NN(\gamma_+ f) = \Upsilon (\bar {\frak N})
$$
strongly in $L^1(0,T)$. We have established that $\Upsilon$ is a compact function.

  \smallskip
$\bullet$
We equip  $\NNN_R$ with the strong topology of $L^1(0,T)$ and we observe that it is a closed and convex set.  From the two above properties of $\Upsilon$, we observe that $\Upsilon: \NNN_R \to \NNN_R$ is continuous and compact.
We may apply the Schauder  fixed point theorem which tells us that there exists $\frak N^* \in \NNN_R$ such that $\Upsilon (\frak N^*) = \frak N^*$. 
The function $F^n := f_{\frak N^*}$ solution to the   equation \eqref{eq:VCkTrunc-Lin},  \eqref{eq:def-Nn},  \eqref{eq:LinearVCkBd},  \eqref{eq:VCkTrunct=0} is such that $\frak N^*  = \NN(\gamma_+ F^n)$.  
  In other words, $F^n$ is a solution to the regularized Voltage-Conductance kinetic equation \eqref{eq:VCkn}, 
\eqref{eq:VCkn-coeff},  \eqref{eq:VCktBd1}, \eqref{eq:VCktBd2}, \eqref{eq:VCktBd0-nonlinear-n}, \eqref{eq:VCkTrunct=0}.

\Black
\subsection{Uniform estimates and passing to the limit in the approximation sequence} 
\label{subsec:passingTOlimit}

We recall the assumption 
$$
C_0 :=  \int_\OO F_0 ( \langle y \rangle^2 + |\log F_0|) dvdy < \infty, 
$$
we define the initial datum 
$$
F_0^n := F_0 \wedge n \, {\bf 1}_{y \le n}
$$
and we consider $F^n \in X_T$ the solution provided by Proposition~\ref{prop:existence-Schauder}  
 to the truncated Voltage-Conductance kinetic equation \eqref{eq:VCkn},  \eqref{eq:VCktBd1}, \eqref{eq:VCktBd2},  \eqref{eq:VCktBd0-nonlinear-n}
 and associated to the above initial datum $F^n_0$. Because \eqref{eq:VCkn} is of the form \eqref{eq:Gal-eq-for-APB}, we may use Propositions~\ref{prop:APB1}, ~\ref{prop:APB2}, ~\ref{prop:APB3}
 and we immediately obtain that for any $T > 0$, there exists a constant $C_T := C(T,C_0)$, and thus independent of $n \ge 1$, such that 
 \bean
 &&
  \sup_{[0,T]} \int_\OO F^n ( |\log F^n| + \langle y \rangle^2)  dydv + \int_0^T a^n _{F^n} \II(F^n) dt  \le C_T
\\
 &&
  \int_0^T\int_0^\infty J_+ \gamma_+ F^n  \bigr\{ \langle y \rangle  +   \log {J_- \over J_+}  + \chi_0   \log \gamma_+ F^n \bigr\} dy dt \le C_T,
 \eean
 where $\chi_0$ is defined in Proposition~\ref{prop:APB3}. 
  In order to pass to the limit in equations  \eqref{eq:VCkn},  \eqref{eq:VCktBd1}, \eqref{eq:VCktBd2},  \eqref{eq:VCktBd0-nonlinear-n} we may (almost) invoke
 Theorem~\ref{theo:stability}. The only difference in the present context comes from the fact that in  \eqref{eq:VCkn} the coefficients $K^n_F$ and $a^n_F$ involved the truncation $\NN^n_{F^n}$ instead of $\NN_{F^n}$, 
 which does not cause any additional difficulty. 
  Indeed, coming back to the proof of  Theorem~\ref{theo:stability}, we still have \eqref{eq:theo-stab-cvgce2}, from what we deduce 
 \beqn\label{eq:CVGCE-Nn-truncBIS}
\NN_n := \NN(\gamma_+ F^n) \wto \NN(\bar\gamma_+) 
\ \hbox{ weakly in } \ L^1(0,T),
\eeqn
and the reverse sense of the Dunford-Pettis Lemma tells us that 
$( \NN_n )$ is equi-integrable, that is 
$$
\sup_{n \ge 1} \int_0^T \NN_n {\bf 1}_{\NN_n \ge A} dt \le \eta(A) \to 0, \quad \hbox{as}\quad A \to \infty.
$$
Writing
$$
0 \le \NN_n - \NN_n \wedge n \le \NN_n {\bf 1}_{\NN_n \ge n},
$$
we then deduce that 
$$
\| \NN_n - \NN_n \wedge n\|_{L^1(0,T)} \le \eta(n) \to 0, \quad \hbox{as}\quad n\to \infty.
$$
This last estimate and \eqref{eq:CVGCE-Nn-truncBIS} together immediately imply  
 $$ 
\NN^n_{F^n}  = \NN(\gamma_+ F^n) \wedge n    \wto \bar \NN := \NN(\bar\gamma_+)  \quad \hbox{weakly in } L^1(0,T),
$$
and that is the only necessary information for finishing the existence proof just as in the proof  of Theorem~\ref{theo:stability}.

\bigskip
\bigskip
\bibliographystyle{plain}

\end{document}